\documentclass[14pt]{article}
\usepackage[a4paper, margin = 3cm]{geometry}
\usepackage{amsmath, amssymb, amsthm, bbm, tikz-cd, url, hyperref}
\usepackage{thmtools}
\usepackage[affil-it]{authblk}

\usepackage{amsmath, amsthm, stmaryrd, amsfonts, cleveref, tikz-cd, enumitem}

\newcommand{\ind}{\leavevmode{\parindent=15pt\indent}}

\theoremstyle{plain} %Theorem, Lemma, Corollary, Proposition, Conjecture, Criterion, Assertion
\newtheorem{theorem}{Theorem}[section]
\newtheorem*{theorem*}{Theorem}
\newtheorem{corollary}[theorem]{Corollary}
\newtheorem*{corollary*}{Corollary}
\newtheorem{proposition}[theorem]{Proposition}
\newtheorem*{proposition*}{Proposition}
\newtheorem{lemma}[theorem]{Lemma}
\newtheorem*{lemma*}{Lemma}

\newtheorem*{fact*}{Fact}

\newtheorem*{conjecture*}{Conjecture}

\newtheorem*{criterion*}{Criterion}

\newtheorem*{assertion*}{Assertion}

\newtheorem*{lem_def*}{Lemma-Definition}

\newtheorem*{prop_def*}{Proposition_Definition}

\newtheorem*{thm_def*}{Theorem-Definition}

\theoremstyle{definition} %Definition, Condition, Problem, Example, Exercise, Algorithm, Question, Axiom, Property, Assumption, Hypothesis
\newtheorem{example}[theorem]{Example}
\newtheorem*{example*}{Example}
\newtheorem{examples}[theorem]{Examples}
\newtheorem*{examples*}{Examples}
\newtheorem{definition}[theorem]{Definition}
\newtheorem*{definition*}{Definition}

\newtheorem*{condition*}{Condition}

\newtheorem*{problem*}{Problem}

\newtheorem*{exercise*}{Exercise}

\newtheorem*{algorithm*}{Algorithm}

\newtheorem*{subroutine*}{Subroutine}

\newtheorem*{question*}{Question}

\newtheorem*{axiom*}{Axiom}

\newtheorem*{property*}{Property}

\newtheorem*{assumption*}{Assumption}

\newtheorem*{hypothesis*}{Hypothesis}

\theoremstyle{remark} %Remark, Note, Notation, Claim, Summary, Acknowledgment, Case, Conclusion
\newtheorem{remark}[theorem]{Remark}
\newtheorem{remarks}[theorem]{Remarks}
\newtheorem*{remark*}{Remark}

\newtheorem*{note*}{Note}

\newtheorem*{scholium*}{Scholium}

\newtheorem*{notation*}{Notation}

\newtheorem*{claim*}{Claim}

\newtheorem*{summary*}{Summary}

\newtheorem*{acknowledgment*}{acknowledgment}

\newtheorem*{acknowledgement*}{acknowledgement}

\newtheorem*{case*}{Case}

\newtheorem*{conclusion*}{Conclusion}
\makeatletter
\def\thmheadbrackets#1#2#3{%
  \thmname{#1}\thmnumber{\@ifnotempty{#1}{ }\@upn{#2}}%
  \thmnote{ {\the\thm@notefont[#3]}}}
\makeatother

\newtheoremstyle{brackets}% Name
  {}% space above
  {}% space below
  {\itshape}% body font
  {}% indent
  {\bfseries}% head font
  {.}% punctuation after head
  { }% space after head (has to be space or dimension!)
  {\thmheadbrackets{#1}{#2}{#3}}% head spec

\theoremstyle{brackets}

\newtheorem*{theorembrackets*}{Theorem}

\makeatletter
\def\thmheadnoparens#1#2#3{%
  \thmname{#1}\thmnumber{\@ifnotempty{#1}{ }\@upn{#2}}%
  \thmnote{ {\the\thm@notefont#3}}}
\makeatother

\newtheoremstyle{noparens}% Name
  {}% space above
  {}% space below
  {\itshape}% body font
  {}% indent
  {\bfseries}% head font
  {.}% punctuation after head
  { }% space after head (has to be space or dimension!)
  {\thmheadnoparens{#1}{#2}{#3}}% head spec

\theoremstyle{noparens}

\newtheorem*{theoremnoparens*}{Theorem}

\usepackage{xparse}
\DeclareDocumentCommand{\newmathcommand}{mO{0}m}{%
  \expandafter\let\csname old\string#1\endcsname=#1
  \expandafter\newcommand\csname new\string#1\endcsname[#2]{#3}
  \DeclareRobustCommand#1{%
    \ifmmode
      \expandafter\let\expandafter\next\csname new\string#1\endcsname
    \else
      \expandafter\let\expandafter\next\csname old\string#1\endcsname
    \fi
    \next
  }%
}

\newcommand{\A}{\mathbb{A}}

\newcommand{\C}{\mathbb{C}}

\newmathcommand{\H}{{\mathbb{H}}}

\newcommand{\K}{\mathbb{K}}
\newmathcommand{\L}{\mathbb{L}}

\newmathcommand{\O}{\mathbb{O}}
\newmathcommand{\P}{\mathbb{P}}

\newcommand{\R}{\mathbb{R}}
\newmathcommand{\S}{\mathbb{S}}

\newcommand{\Z}{\mathbb{Z}}
\newmathcommand{\deg}{\operatorname{deg}}

\DeclareMathOperator{\End}{End}

\newmathcommand{\Big}{{\mathrm{Big}}}

\DeclareMathOperator{\FJ}{FJ}
\DeclareMathOperator{\FSJ}{FSJ}
\DeclareMathOperator{\FJP}{FJP}

\DeclareMathOperator{\JB}{JB}

\DeclareMathOperator{\rk}{rk}

\newmathcommand{\dv}{\operatorname{div}}
\newmathcommand{\inw}{\operatorname{int}}

\newmathcommand{\epi}{\operatorname{epi}}
\newmathcommand{\diam}{\operatorname{diam}}
\newmathcommand{\diag}{\operatorname{diag}}
\newmathcommand{\regn}{\operatorname{regn}}
\newmathcommand{\relinw}{\operatorname{relint}}

\newmathcommand{\char}{\operatorname{char}}

\DeclareMathOperator{\Id}{Id}
\DeclareMathOperator{\id}{id}

\newmathcommand{\unr}{\mathrm{unr}}

\newmathcommand{\Un}{\operatorname{Un}}

\newmathcommand{\sl}{\mathfrak{s}\mathfrak{l}}

\newmathcommand{\un}{\mathfrak{u}}
\newmathcommand{\su}{\mathfrak{s}\mathfrak{u}}

\title{On the Shirshov--Cohn theorem for JB-algebras}
\author{Mark Roelands%
\thanks{Email: \texttt{m.roelands@math.leidenuniv.nl}}}
\affil{Mathematical Institute, Leiden University, 2300 RA Leiden,
The Netherlands}

\author{Samuel Tiersma%
\thanks{Email: \texttt{s.j.tiersma@math.leidenuniv.nl}}}
\affil{Mathematical Institute, Leiden University, 2300 RA Leiden,
The Netherlands}

\begin{document}

\maketitle

\begin{abstract}
\noindent It is shown that a JB-algebra which can be generated by the union of two of its associative Jordan subalgebras is a JC-algebra, hence special. A similar refinement of Macdonald's principle for JB-algebras is obtained. Moreover, we prove that the free unital JB-algebra generated by $n$ projections is a JC-algebra if and only if $n\in \{1,2,3\}$. Finally, we give an explicit description of the free unital JB-algebra generated by two projections paralleling the Raeburn-Sinclair theorem for C*-algebras.
\end{abstract}


{\small {\bf Subject Classification: 	17C05, 17C40, 17C10}

\section{Introduction}
%The standard formulation of quantum mechanics models physical observables as the self-adjoint operators on a complex Hilbert space $H$. Although the real vector space $B(H)_{sa}$ is not closed composition, P. Jordan \cite{Jordan33} observed that it is closed under the commutative product $S \circ T := \frac{1}{2}(ST+TS)$, which satisfies the \emph{Jordan identity} $S \circ (T \circ S^2) = (S \circ T) \circ S^2$. In an attempt to discover alternative models for quantum mechanics, P. Jordan, J. von Neumann and E. Wigner axiomatized the properties of the Jorda 
%Later, E. Alfsen, F. Shultz and E. Størmer introduced in \cite{ASS78} a class of Jordan algebras with a compatible Banach space structure called \emph{JB-algebras}. The fundamental example of a JB-algebra is the real vector space $B(H)_{sa}$ of bounded self-adjoint operators on a complex Hilbert space $H$ endowed with the Jordan product $x \circ y = \frac{1}{2}(xy+yx)$. A JB-algebra is called a \emph{JC-algebra} if it embeds in $B(H)_{sa}$. It is an intriguing question how far 
Jordan algebras were introduced in the 1930s as an alternative algebraic framework for the mathematical formulation of quantum
mechanics \cite{JNW}. Since, Jordan algebras have been extensively studied for their rich connections with other fields of mathematics, such as Lie theory \cite{Jac71}, \cite{FaFe77}, \cite{SpVe}, the study of symmetric spaces   \cite{Sa80}, \cite{Up85}, \cite{Kaup02}  and operator theory \cite{Topping65}, \cite{EffSto79}, \cite{AHOS}, \cite{Up87}.\\
\ind For each associative real algebra $A$ and linear subspace $J \subset A$ which is closed under taking squares, $J$ is a Jordan algebra under the product $x \circ y := \frac{1}{2}(xy+yx)$. A Jordan algebra is called \emph{special} if it arises in this way from an associative algebra, otherwise it is called \emph{exceptional}. A.A. Albert proved that the \emph{Albert algebra} $\A := M_3(\O)_{sa}$ of self-adjoint $3\times 3$-matrices with entries in the nonassociative division algebra $\O$ of octonions is exceptional \cite{Albexc}. Later, Glennie gave an elegant proof of this fact by exhibiting two Jordan identities satisfied by all special Jordan algebras but not by the Albert algebra \cite{Glennie}, initiating a quest for other such identities \cite{Thedy87}, \cite{Zhev}, \cite{Sverchkov}. The distinction between special and exceptional Jordan algebras forms a central theme in the structure theory of Jordan algebras \cite{ZelP1}, \cite{JacStrJA}, \cite{ZelPII}. \\ %Later, Glennie gave an elegant proof of this fact by exhibiting two Jordan identities satisfied by all special Jordan algebras but not by the Albert algebra.
%\ind After the discovery of the exceptional Albert algebra, it has become a central theme in the theory of Jordan algebras to describe their structure in terms of associative structures. In particular, people have seeked condition which guarantee that a Jordan algebra is either special or exceptional. Jordan identities have been intensively studied because they shed light on this question.\\\cite{ZelP1}
\ind There are two classical results that address this theme: the Shirshov--Cohn theorem and Macdonald's principle. The Shirshov--Cohn theorem asserts that every Jordan algebra which can be generated by two elements (and $1$) is special. According to Macdonald's principle, if a Jordan identity involving three variables is linear in one of these and is satisfied by all special Jordan algebras, then this identity is satisfied by all Jordan algebras. Macdonald's principle has been recognized as an indispensable tool \cite[p. 458]{McC04} in the development of the theory of Jordan algebras, allowing the efficient verification of many Jordan identities by reducing to the case of a special Jordan algebra and then calculating in the ambient associative structure.\\
\ind In this paper we revisit these results in the setting of JB- and JBW-algebras. The prototypical example of a JB--algebra is the self-adjoint part $A_{sa}$ of a C*-algebra $A$ endowed with the Jordan product $x \circ y := \frac{1}{2}(xy+yx)$. Note that while the C*-algebra product of two self-adjoint elements is generally not self-adjoint, the Jordan product does yield an intrinsic algebraic structure on the ordered vector space $A_{sa}$. For this reason, JB-algebras provide a versatile algebraic framework to study order theory in a C*-algebra.\\
%In fact, the unital JB-algebras can order-theoretically be characterized as the order unit spaces for which the interior of the positive cones admits an order-anti-isomorphism that is homogeneous of degree $-1$ \cite{OCJB}.  \\
\ind A JB-algebra is called a \emph{JC-algebra} if it embeds in $B(H)_{sa}$ for some complex Hilbert space $H$. The Shirshov--Cohn theorem for JB-algebras says that each JB-algebra which can be generated by two elements (and $1$) is a JC-algebra. We generalize this theorem by replacing the two elements in its statement with two associative Jordan subalgebras.

\begin{theorem}\label{11}
    A JB-algebra is a JC-algebra if it can be generated by the union of two of its associative Jordan subalgebras. 
\end{theorem}

This result is proved as \Cref{ShirCohn2}. Since the Albert algebra can be generated by three elements \cite[Cor. 2]{AlbPaige}, it is perhaps surprising that a JB-algebra generated by \emph{three projections} is necessarily a JC-algebra.

\begin{theorem}\label{12}
    A JB-algebra is a JC-algebra if it can be generated by three projections (possibly together with the unit element). 
\end{theorem}

We prove this result as \Cref{ShirCohn3}. The analogous assertions for JBW-algebras also hold. We give a similar refinement of Macdonald's principle for JB-algebras. Recall that elements $x_1, x_2$ in a JB-algebra $A$ are said to \emph{operator commute} if $x_1 \circ (x_2 \circ a) = x_2 \circ (x_1 \circ a)$ for all $a\in A$.

%, which now allows for two tuples of operator commuting elements and a third element as is made explicit below.

\begin{theorem}\label{13}
    Let $F(X_1,\dots,X_n, Y_1,\dots,Y_m, Z)$ be a real Jordan polynomial which has degree at most $1$ in $Z$ and is satisfied by all JC-algebras. Let $A$ be a JB-algebra. Suppose that $(x_1, \dots, x_n)$ and $(y_1,\dots,y_m)$ are two tuples of mutually operator commuting elements in $A$. Then for every $z\in A$ it holds that $F(x_1,\dots, x_n, y_1, \dots, y_m, z) = 0$. 
\end{theorem}

This result will be established as \Cref{McDImp}. 
Each of these results is proved by making a reduction using \cite[Thm. 9.5]{ASS78} to the case of the Albert algebra and analyzing the Jordan subalgebras of $M_3(\O)_{sa}$ having certain generating sets.\\
\ind Moreover, in \Cref{FJ2Sec} we describe the free unital JB-algebra generated by two projections as the JB-algebra of all continuous functions $f \colon [0,1] \to M_2(\mathbb{R})_{sa}$ into the $3$-dimensional spin factor such that $f(0)$ and $f(1)$ are diagonal matrices. This result is the JB-analogue of a theorem due to Raeburn--Sinclair for C*-algebras \cite{RaeSin}.\\
\ind In order to make the article accessible to a broad spectrum of readers, \Cref{PrelSec} includes a relatively self-contained discussion of the pertitent concepts and results in the theory of Jordan algebras and JB-algebras.\\
%\ind We conclude this introduction by pointing out an algebraic direction for future research. Zelmanov proved in \cite[Thm. 2]{ZelPII} that each prime nondegenerate Jordan algebra over a commutative ring $R$ with $\frac{1}{2}$ is either special or a form of a split Albert algebra. It may be interesting to investigate whether the (identity-)speciality results in \Cref{11} and \Cref{12} can be obtained for (quadratic) Jordan algebras by similar techniques as in this paper, i.e.\ by making a reduction to the case of an Albert algebra and analyzing its subalgebras.

%A JB-algebra (resp.\ JBW-algebra) is called a \emph{JC-algebra} (resp.\ \emph{JW-algebra}) if it embeds in $B(H)_{sa}$ as a norm closed (resp.\ ultraweakly closed) Jordan subalgebra, for some complex Hilbert space $H$. The Shirshov--Cohn theorem for JB(W)-algebras says that each JB- or JBW-algebra which can be generated by two elements (and the identity) is a JC- (resp.\ JW-algebra). Perhaps surprisingly, \emph{three} projections generate a JC-algebra as well.

%The structure of the paper is as follows. \Cref{PrelSec} is preliminary and provides the necessary background on the theory of Jordan algebras and JB-algebras. We devote Section 3 to the generalization of the Shirshov--Cohn theorem and studying JB-algebras generated by projections. In Section 4 we prove the generalized Macdonald's principle.

\section{Preliminaries}\label{PrelSec}

\subsection{Jordan algebras}\label{JdsubSect}

A real \emph{Jordan algebra} is a real vector space $J$ equipped with a commutative (but not necessarily associative) bilinear product $(x, y) \mapsto x \circ y$  which satisfies the following \emph{Jordan identity} for all $x,y\in J$:
\begin{equation}
    x \circ (y \circ x^2) = (x \circ y) \circ x^2.
\end{equation}
A real linear subspace $E$ of $J$ is called a \emph{Jordan subalgebra} of $J$ if it is closed under the product $\circ$, and $E$ is called a (Jordan) \emph{ideal} of $J$ if $x \circ y \in E$ for all $x\in E$ and $y \in J$.\\
\ind We say $J$ is \emph{unital} if it has a \emph{unit element} $1_J$ such that $1_J \circ x = x$ for all $x\in J$. If $J$ is not unital, then we may form its \emph{unitization} $\tilde{J} := \R1_J \oplus J$ with the product $$(\lambda 1_J + x) \circ (\mu 1_J + y) = \lambda\mu 1_J + (\lambda y + \mu x + x \circ y)$$ 
for $\lambda,\mu \in \R$ and $x,y\in A$; then $\tilde{J}$ is a unital Jordan algebra containing $J$ as an ideal.\\
\ind Let $K$ be a second Jordan algebra. A linear map $\phi \colon J \to K$ is called a \emph{Jordan homomorphism} if $\phi(x \circ y) = \phi(x) \circ \phi(y)$. If $J$ and $K$ are unital with unit elements $1_J$ resp.\ $1_K$, then a Jordan homomorphism $\phi\colon J \to K$ is called \emph{unital} if $\phi(1_J)=1_K$.
% %In the absence of $\frac{1}{2}$, t\emph{quadratic} Jordan algebras, but this will not concern us.

\begin{example}
Let $A$ be an associative real algebra. Let $J \subset A$ be a real linear subspace which is closed under taking squares. Then $J$ is a real Jordan algebra under the \emph{Jordan product}
\begin{equation}\label{Jprod}
    x \circ y := \tfrac{1}{2}(xy+yx).
\end{equation}
\end{example}

\begin{example}\label{Aex} Let $\O = \R1 \oplus \sum_{i=1}^7 \R e_i$ be the non-associative real division algebra of Cayley numbers, or octonions. Equip $\O$ with its standard involution, i.e.\ $e_i^* = -e_i$ for all $1\le i\le 7$. The \emph{Albert algebra} is the algebra of $3\times 3$-self adjoint matrices with entries in $\O$ and denoted
$$\A := M_3(\O)_{sa}.$$
The proof that $\A$ is indeed a Jordan algebra is involved and may be found in \cite[Cor. V.2.6]{FaKor}.
\end{example}

In the remainder of this section and the next section, we will implicitly assume that all algebras, vector spaces and homomorphisms are defined over $\R$. However, the real field $\R$ may throughout be replaced by an arbitrary commutative associative ring $R$ with $1$ in which $2$ is invertible.\\
\\
\textbf{The Peirce decomposition.} A useful tool in the structure theory of Jordan algebras is the \emph{Peirce decomposition} with respect to a supplementary system of orthogonal elements. Let $J$ be a Jordan algebra and let $\underline{e} = (e_1, e_2, \ldots, e_n) \in J$ be a system of idempotent elements such that $e_1 + e_2 + \ldots + e_n = 1$. Each multiplication operator $T_{e_r}\colon J \to J$, $T_{e_r}(x) = e_r \circ x$ satisfies the polynomial $X(X-\frac{1}{2})(X-1)=0$ by \cite[Prop. III.1.3]{FaKor}
and commutes with $T_{e_s}$ for $s\neq r$ by \cite[Lemma IV.1.4]{FaKor}. Define for all $1\le r\neq s\le n$ the \emph{Peirce eigenspaces} by
$$J_{rr}(\underline{e}) := \{x\in J: e_r \circ x = x\}, \quad J_{rs}(\underline{e}) := \{x\in J: e_r \circ x = e_s \circ x = \tfrac{1}{2}x\}.$$
Since $T_{e_1} + \dots + T_{e_n} = \Id_J$, we obtain the \emph{Peirce decomposition of} $J$ relative to $\underline{e}$, 
\begin{align}\label{PcJ}
J = \bigoplus_{1\le r\le s\le n} J_{rs}(\underline{e}),
\end{align}
in which the various Peirce eigenspaces satisfy the following multiplication rules \cite[Thm. IV.2.1]{FaKor}:
\begin{alignat}{2} 
J_{rs} \circ J_{rs} &\subset J_{rr} + J_{ss}, \label{Pr1}\\
J_{rs} \circ J_{st} &\subset J_{rt}, &&\quad \text{if } r \neq t, \label{Pr2}\\
J_{rs} \circ J_{tu} &= \{0\}, &&\quad \text{if } \{r,s\} \cap \{t,u\} = \emptyset. \label{Pr3}
\end{alignat}
}

\subsection{Identities and special Jordan algebras}
A Jordan algebra $J$ is called \emph{special} if there exists an associative algebra $A$ and an injective linear map $\alpha \colon J \to A$ such that for all $x,y \in A$ one has $$\phi(x\circ y) = \tfrac{1}{2}(\alpha(x)\alpha(y)+\alpha(y)\alpha(x)).$$
A Jordan algebra which is not special is called \emph{exceptional}. Albert proved in \cite{Albexc} that the Albert algebra $\A$ in \Cref{Aex} is exceptional. More examples are furnished by free Jordan algebras.

\begin{definition}\label{FJS} Let $S$ be a set. The \emph{free unital Jordan algebra} on the set $S$ is a unital Jordan algebra $\FJ(S)$ together with a map of sets $\iota_S\colon S \to J$ satisfying the following universal property. For each unital Jordan algebra $J$ and map of sets $\alpha \colon S \to J$ there exists a unique unital Jordan homomorphism $\psi\colon \FJ(S) \to J$ such that $\psi(\iota(s)) = \alpha(s)$ for all $s\in S$.
\end{definition}

The free unital Jordan algebra exists and the pair $(\FJ(S), \iota_S)$ is unique up to a unique isomorphism \cite[Ch. 1, Sect. 6]{JacStruRep}. A similar free object in the category of \emph{special} unital Jordan algebras exists and is constructed explicitly as follows \cite[p. 684]{LewMcC}. Denote by $\R\{S\}$ the free unital associative algebra on the set $S$, consisting of all real polynomials in the noncommuting variables $X_s$ for $s\in S$. We endow $\R\{S\}$ with the Jordan product \eqref{Jprod}. Let $\FSJ(S)$ be the Jordan subalgebra of $\R\{S\}$ generated by $\{X_s\colon s \in S\}$, and define $\iota_S\colon S \to \FSJ(S)$, $\iota_S(s) := X_s$. Then, given a unital special Jordan algebra $J$, say embedded in the unital associative algebra $A$, and a map of sets $\alpha\colon S \to J$, there exists a unique unital algebra homomorphism $\tilde{\psi}\colon \R\{S\} \to A$ such that $\tilde{\psi}(X_s) = \alpha(s)$ for every $s\in S$, which restricts to a unique unital Jordan homomorphism $\psi\colon \FSJ(S) \to J$ such that $\psi \circ \iota_S = \alpha$. Therefore, the pair $(\FSJ(S), \iota_S)$ is the \emph{free unital special Jordan algebra} on the set $S$, in the sense that it satisfies the universal property in \Cref{FJS} but with $J$ restricted to the class of unital \emph{special} Jordan algebras.\\

For each positive integer $n$, denote $[n] := \{1, 2, \ldots, n\}$. A (real) \emph{Jordan polynomial} $F$ is an element of $\FJ([n])$ for some $n\ge 1$. Let $J$ be a unital Jordan algebra. We say that $J$ \emph{satisfies the Jordan polynomial} $F \in \FJ([n])$ if for all elements $x_1, \ldots, x_n \in J$ the unique unital Jordan homomorphism $\psi \colon \FJ([n]) \to J$ such that $\psi(\iota_{[n]}(k)) = x_i$ for all $1\le k\le n$ satisfies $\psi(F) = 0$. It is easy to see that for the identity $F = 0$ to be satisfiable by some unital Jordan algebra, it is necessary that the constant term of $F$ be zero. It follows that $J$ satisfies $F$ if and only if $\tilde{J}$ satisfies $F$. Without creating ambiguity, we will say that a non-unital Jordan algebra $K$ \emph{satisfies} $F$ if its unitization $\tilde{K}$ satisfies $F$. If $G, H \in \FJ([m])$ are Jordan polynomials, we will also say that a Jordan algebra \emph{satisfies the identity} $G = H$ if it is satisfies the Jordan polynomial $F := G - H$.\\
\ind There exist identities satisfied by all special Jordan algebras, but not by all Jordan algebras. Two identities of total degree $8$ resp. $9$ were exhibited by Glennie in \cite[Thm 4(a)]{Glennie} and read
\begin{align}\label{G8}
G_8: H_8(X,Y,Z) &=H _8(Y, X, Z),\\
H_8(X, Y, Z) &:= \{U_XU_Y(Z), Z, X\circ Y\} - U_XU_YU_Z(X\circ Y). \notag
\end{align}
resp.\ 
\begin{align}\label{G9}
G_9: H_9(X,Y,Z) &=H _9(Y, X, Z),\\
H_9(X, Y, Z) &:= 2(U_XZ) \circ U_{Y,X}U_Z(Y^2) - U_XU_ZU_{X,Y}U_Y(Z). \notag
\end{align}
%In the free special Jordan algebra $\FSJ(X, Y, Z) \subset \R\{X,Y,Z\}$ one has
%$$H_8(X, Y, Z) = \{X, Y, Z, Y, X, Z, X, Y\}$$...
Glennie's identities are satisfied by $\FSJ(X, Y, Z)$ but not by the Albert algebra $\A$.\\
\ind A Jordan algebra $J$ is called \emph{identity-special} if it satisfies each identity satisfied by all special Jordan algebras. A unital Jordan algebra $J$ is identity-special if and only if $J$ is a homomorphic image of a (free) special Jordan algebra. (If $J$ is identity-special then, with $S$ equal to the set underlying $J$, the surjective unital Jordan homomorphism $\psi\colon \FJ(S) \to \tilde{J}$ such that $\psi \circ \iota_S = \Id_J$ factorizes through $\FSJ(S)$.) Glennie's identity shows that the Albert algebra $\A$ is not identity-special, and thus gives another proof of the Albert algebra's exceptionality.\\
\ind P.M. Cohn showed that a quotient $\FSJ(S)/K$ of the free special Jordan algebra $\FSJ(S)$ by a Jordan ideal $K$ is special if and only if $K = \langle K\rangle \cap \FSJ(S)$, where $\langle K\rangle \subset \R\{S\}$ denotes the ideal of the associative algebra $\R\{S\}$ generated by $K$. He used this speciality criterion to give the following example of an identity-special but not special Jordan algebra \cite[Scholium 6.1]{Cohn_1954}.

\begin{example}\label{isnots} Let $K$ be the Jordan ideal of $\FSJ([3])$ generated by $X_1^2 - X_2^2$. Then the Jordan algebra $\FSJ([3])/K$ is identity-special but not special.
\end{example}

We now come to two foundational results in the theory of Jordan algebras. A.I. Shirshov showed in \cite[Thm. 2]{Shi56} that $\FJ([2])$ is a special Jordan algebra. Cohn showed in \cite[Thm. 5.2]{Cohn_1954} that every homomorphic image of $\FSJ([2])$ is again special. The Shirshov--Cohn theorem combines these results. 

\begin{theorem}[Shirshov--Cohn]
Each Jordan algebra which can be generated by two elements (possibly together with the unit element) is special.
\end{theorem}
\begin{proof}
    See \cite[Thm. 10]{JacStruRep}.
\end{proof}

Shirshov's theorem may be stated by saying that the canonical map $\FJ([2]) \to \FSJ([2])$ is an isomorphism. I.G. Macdonald proved in \cite{McD} that each non-zero element of the kernel of $\FJ([3]) \to \FSJ([3])$ has degree at least $2$ in each of the generators. This result has come to be known as \emph{Macdonald's principle}.

\begin{theorem}[Macdonald]
Let $F(X, Y, Z)$ be a Jordan polynomial of degree at most $1$ in $Z$. If $F$ is satisfied by all special Jordan algebras, then $F$ is satisfied by all Jordan algebras.
\end{theorem}

\begin{proof}
We give an outline of the proof given by Jacobson in \cite{JacMcD}. Given a subalgebra $B$ of a Jordan algebra $J$, we denote by $M_B(J)$ the \emph{multiplication subalgebra} of $\End(J)$ generated by the multiplication operators $L(b)$ for $b\in B$. By Shirshov's theorem, $\FJ([2]) \cong \FSJ([2])$ may be considered as a Jordan subalgebra of both $\FJ([3])$ and $\FSJ([3])$.
The canonical surjective homomorphism $\FJ([3]) \to \FSJ([3])$ induces a surjective algebra homomorphism
\begin{equation}\label{sigma}
\sigma\colon M_{\FJ([2])}(\FJ[3]) \to M_{\FJ([2])}(\FSJ[3]).
\end{equation}
Macdonald's principle is equivalent to the assertion that $\sigma$ is injective (confer \cite[B.4.1]{McC04}). Jacobson proved that $\sigma$ is indeed an isomorphism by presenting both multiplication algebras in \eqref{sigma} using the same generators and relations.  
\end{proof}

We conclude this section by an elementary lemma, which will prove useful later.

\begin{lemma}\label{tauLem} Let $A$ be a real Jordan algebra. Let $\tau$ be a linear topology or linear convergence structure on $A$ such that multiplication on $A$ is separately $\tau$-continuous. Let $F$ be a Jordan polynomial and let $D$ be a Jordan subalgebra of $A$ satisfying $F$. Then the $\tau$-closure $\overline{D}^{\tau}$ of $D$ is a real Jordan subalgebra of $A$ satisfying $F$.
\end{lemma}
\begin{proof}
First assume that $\tau$ is a linear topology. Then $\overline{D}^{\tau}$ is readily seen to be a Jordan subalgebra of $A$. The Jordan polynomial $F$ is the sum of homogeneous Jordan polynomials
$$F(X_1,\ldots, X_n) = \sum_{d\in \Z_{\ge 0}^n} F_d(X_1, \ldots, X_n)$$
where for each $d = (d_1, \ldots, d_n) \in \Z_{\ge 0}^n$ and $1\le i\le n$, the Jordan polynomial $F_d$ is homogeneous of degree $d_i$ in $X_i$. Then for all $\lambda_1,\ldots,\lambda_n \in \R$ and $x_1,\ldots,x_n \in D$ we have
$$0 = F(\lambda_1 x_1,\ldots,\lambda_n x_n) = \sum_{d\in \Z_{\ge 0}^n} \lambda_1^{d_1}\cdots \lambda_n^{d_n} F_d(x_1,\ldots,x_n).$$
Since $\R$ is an infinite field, it follows that $D$ satisfies each $F_d$ (and that $F_0 = 0$ if $D\neq 0$). Fix $d\in \Z_{\ge 0}^n$ such that $m := \sum_{i=1}^n d_i > 0$. Let $\tilde{F}_d(W_1,\ldots,W_m)$ be the complete linearization of $F_d$, so $\tilde{F}_d$ is linear in each of its $m = \sum_{i=1}^n d_i$ variables and for all $x_1,\ldots,x_n \in A$ we have
\begin{equation}\label{Ft}
\tilde{F}_d(
    \underbrace{x_1, x_1, \ldots, x_1}_{d_1 \text{ times}}, 
    \underbrace{x_2, x_2, \ldots, x_2}_{d_2 \text{ times}}, 
    \ldots, 
    \underbrace{x_n, x_n, \ldots, x_n}_{d_n \text{ times}}) = F(x_1,\ldots,x_n).
\end{equation}
By induction on $0\le k\le m$ we prove that for all $w_1,\ldots,w_k \in \overline{D}^{\tau}$ and $w_{k+1},\ldots,w_m \in D$ we have $\tilde{F}_d(w_1, \ldots, w_m) = 0$. The induction base holds since $D$ satisfies $\tilde{F}_d$. Let $0< l\le m$ and assume the statement is proved for $k = l-1$. Let $w_1,\ldots,w_l \in \overline{D}^{\tau}$ and $w_{l+1},\ldots,w_m \in D$ be arbitrary. Choose a net $(a_\lambda)_\lambda$ in $D$ such that $a_\lambda \stackrel{\tau}{\to} w_k$. Since multiplication is separately $\tau$-continuous and $\tilde{F}_d$ is linear in $W_{k}$, we have by the inductive hypothesis
$$\tilde{F}_d(w_1, \ldots, w_m) = \lim_\lambda \tilde{F}_d(w_1, \ldots, w_{l-1}, a_\lambda, w_{l+1}, \ldots, w_m) = \lim_\lambda 0 = 0.$$
This proves the statement for $k=l$ and completes the inductive step. Now by induction the statement holds for $k=m$, whence $\tilde{F}_d$ is satisfied by $\overline{D}^{\tau}$. But then $\overline{D}^{\tau}$ also satisfies $F_d$ in view of \eqref{Ft}. We conclude that $F = \sum_{d} F_d$ is satisfied by $\overline{D}^{\tau}$.\\
\ind We now return to the general case that $\tau$ is a linear convergence structure. Using the same argument, one shows by transfinite recursion on an ordinal $\alpha$ that elements of the $\alpha$-adherence $\overline{D}^{\tau,\alpha}$ of $D$ satisfy each $\tilde{F}_d$ and that $\overline{D}^{\tau,\alpha}$ is a Jordan subalgebra of $A$ if $\alpha$ is a limit ordinal. Since the $\tau$-closure $\overline{D}^{\tau}$ of $D$ is equal to the $\alpha$-adherence $\overline{D}^{\tau,\alpha}$ for a sufficiently large limit ordinal $\alpha$, the proof is complete.
\end{proof}
%and if $w_1 = w_2 = \ldots = w_{d_1} = x_1$, $w_{d_1+1} = \ldots = w_{d_1+d_2} = x_2$, $\ldots$, $w_{d_{n-1}+1} = \ldots = w_{m} = x_n$, then $\tilde{F}_d(w_1, \ldots, w_m) = F(x_1, \ldots, x_n)$.

\begin{example} An example of a linear convergence structure that is not topological to which \Cref{tauLem} applies is given by order convergence on a unital JB-algebra. Indeed, the argument used in \cite[Prop. 2.4]{AS03} to prove that multiplication on a JBW-algebra is separately $\sigma$-weakly continuous is adapted without difficulty to show that the multiplication on a unital JB-algebra is separately order continuous.
\end{example}

\subsection{JB-algebras}
A \emph{JB-algebra} is a real Jordan algebra $J$ which is complete in a norm $\lVert \cdot \rVert$ satisfying for all $x,y \in J$:
\begin{enumerate}[label={\upshape(\arabic*)}]
    \item $\lVert x\circ y\rVert \le \lVert x\rVert \lVert y\rVert$;
    \item $\lVert x\rVert^2 = \lVert x^2\rVert$;
    \item $\lVert x^2\rVert \le \lVert x^2+y^2\rVert$.
\end{enumerate}
The set of squares $$J_+ := \{x^2: x \in J\}$$
is a closed convex cone in $J$, which determines an Archimedean partial order on $J$ given for $x,y\in J$ by
$$x \ge y :\iff x-y \in J_+.$$
Axiom (3) thus asserts that the norm $\lVert\cdot\rVert$ is monotone on $J_+$. As is the case for C*-algebras, there is at most one norm on a real Jordan algebra making it a JB-algebra.\\
\ind A \emph{JB-subalgebra} of $J$ is a closed linear subspace $E$ of $J$ which is closed under the product of $J$. The smallest JB-subalgebra of $J$ containing a given set $S\subset J$ is denoted $\JB(S)$ and called the \emph{JB-algebra generated by }$S$. The JB-algebra generated by a single element $x$ and possibly the unit element is an associative Banach algebra.\\
%An element $x \in J$ is called \emph{Jordan invertible} if there exists $y\in J$ such that $x \circ y = 1_J$ and $x^2 \circ y = x$,
\ind Let $J$ be a JB-algebra with unit element $1_J$. An element $x \in J$ is called \emph{invertible} if $x$ is invertible in the associative Banach algebra $\JB(1,x)$. The \emph{spectrum} of $x$ is the set
\begin{equation}
\sigma(x) := \{\lambda \in \R: \lambda\text{ is not Jordan invertible in }J\}.
\end{equation}
Then according to the spectral theorem in JB-algebras, there is an isomorphism
$$\JB(1,x) \cong C(\sigma(x), \R),$$
sending $x$ to the identical function $\id_{\sigma(x)}: t \mapsto t$. Hence, the norm on $J$ coincides with the spectral radius norm:
\begin{equation}
    \lVert x\rVert = \sup\{|\lambda|: \lambda \in \sigma(x)\}.
\end{equation}
Moreover, $1_J$ is an Archimedean order unit, i.e.\
for all $x\in J$ there exists $\lambda\ge 0$ with $x\le \lambda 1_J$, and the norm on $J$ also coincides with the associated \emph{order unit norm} (cf.\ \cite[Prop. 3.3.10]{HOSt84}):
\begin{equation}
    \lVert x\rVert = \sup\{\lambda \ge 0: -\lambda 1_J \le x \le \lambda 1_J\}.
\end{equation}
If $H$ is a non-unital JB-algebra, then it was shown by Behncke \cite{Behncke} that the unitization $\tilde{H} = \R 1_H \oplus H$ is a JB-algebra in the spectral radius norm (see also \cite[Thm. 3.3.9]{HOSt84}).

\begin{examples}\label{JBEgs}
(1) The associative JB-algebras coincide, up to isomorphism, with spaces $C_0(X, \R)$ for some locally compact Hausdorff space $X$.\\
(2) Each C*-algebra $A$ gives rise to a JB-algebra $A_{sa}$, by endowing the real vector space $A_{sa} := \{x\in A: x = x^*\}$ of self-adjoint elements with the \emph{Jordan product}
\begin{equation}
    x \circ y := \tfrac{1}{2}(xy+yx)
\end{equation}
and the induced norm.\\
(3) The Albert algebra is a JB-algebra under the spectral norm.\\
(4) Let $(H, \langle \cdot,\cdot\rangle)$ be a real Hilbert space. We endow $V := \R \oplus H$ with the product
$$(\lambda, \xi) \circ (\mu, \eta) = (\lambda \mu, \lambda \eta + \mu \xi + \langle \xi, \eta\rangle).$$
and the norm $\lVert (\lambda, \xi)\rVert = |\lambda| + \langle \xi,\xi\rangle^{1/2}$. Then $V$ is a JB-algebra called a \emph{spin-factor}.
\end{examples}

%\emph{spectral norm}
%$$\lVert x \rVert := \sup \{|\lambda|: \lambda \in \R, \lambda - x\text{ is not Jordan invertible in }\A\}.$$
%\\
%\ind If $A$ is a finite-dimensional \emph{formally real} (i.e. $x^2+y^2=0$ implies $x=y=0$ for all $x,y\in A$) Jordan algebra, then $A$ is a JB-algebra under the spectral norm
%$$\lVert x \rVert := \sup \{|\lambda|: \lambda \in \R, \lambda - x\text{ is not invertible in }A\}.$$
%In this way, finite dimensional JB-algebras are in one-to-one correspondence with the finite-dimensional formally real Jordan algebras, also known as \emph{Euclidean Jordan algebras}. The Albert algebra $\A := M_3(\O)_{sa}$ is formally real, hence $\A$ is a JB-algebra for the spectral norm.\\
Let $A$ be a JB-algebra. We say $A$ is a \emph{JC-algebra} if $A$ is isometrically isomorphic to a closed Jordan subalgebra of $B(H)_{sa}$ for some complex Hilbert space $H$. We call $A$ \emph{purely exceptional} if $A$ admits no non-zero homomorphisms into a JC-algebra. Each JB-algebra $A$ has a unique closed ideal $J$ such that $A/C$ is a JC-algebra and $J$ is purely exceptional \cite[Thm. 4.19]{AS03}.\\
\ind Alfsen--Shultz--Størmer established in \cite{ASS78} the following Gelfand--Naimark type theorem for JB-algebras. To state it, note that for an index set $I$ and a JB-algebra $B$, the $\ell^{\infty}$ direct sum
$$\ell^{\infty}(I, B) := \{f\colon I \to B \text{ with }  \textstyle{\sup_{i\in I} \lVert f(i)\rVert < \infty}\}$$
with the supremum norm is again a JB-algebra under the pointwise Jordan product.

\begin{theorem}[Gelfand--Naimark theorem for JB-algebras]\label{GNJB} Each JB-algebra is isomorphic to a closed Jordan subalgebra of $B(H)_{sa} \oplus \ell^{\infty}(I, \A)$ for some complex Hilbert space $H$ and index set $I$.
\end{theorem}

\Cref{GNJB} readily yields that each purely exceptional JB-algebra is isomorphic to a JB-subalgebra of $\ell^{\infty}(I, \A)$ for a suitable index set $I$. Thus, the Albert algebra may be said to be the progenitor of all purely exceptional JB-algebras. The following theorem, which combines \cite[Lemma 9.4 and Theorem 9.5]{ASS78}, is a consequence of this fact.
%that each JBW-algebra is isomorphic to $M \oplus C(K; \A)$ for some JW-algebra $M$ and compact Hausdorff space $K$. Here $C(K; \A)$ is the space of continuous functions of $K$ into $\A$ given the supremum norm and the pointwise Jordan product. By embedding a JB-algebra into its bidual, it follows that each JB-algebra is isomorphic to a closed Jordan subalgebra of $C(K; \A) \oplus B(H)_{sa}$ for some compact Hausdorff space $K$ and complex Hilbert space $H$. 

\begin{theorem}\label{JCThm} Let $A$ be a JB-algebra. Then the following are equivalent:
\begin{enumerate}[label={\upshape(\arabic*)}]
\item $A$ is a JC-algebra;
\item $A$ is a special Jordan algebra;
\item $A$ is an identity-special Jordan algebra;
\item $A$ satisfies Glennie's identity $G_8$ (or $G_9$);
\item there does not exists a surjective Jordan homomorphism $\pi\colon A \to \A$ onto the Albert algebra;
\end{enumerate}
\end{theorem}
\begin{proof}
(1) $\implies$ (2) $\implies$ (3) $\implies$ (4) is immediate, while (4) $\implies$ (5) holds since $\A$ satisfies neither $G_8$ nor $G_9$. The remaining implication (5) $\implies$ (1) is proved in \cite[Cor. 4.20]{AS03}.
\end{proof}

Notwithstanding the equivalence of (1) and (2) in the above theorem, the class of JC-algebras is strictly smaller than the class of special real Jordan algebras. Indeed, the universal property of the free special Jordan algebra on a nonempty $S$ implies there does not exist a norm that makes $\FSJ(S)$ into a JB-algebra, because Jordan homomorphism between JB-algebras are automatically contractive \cite[Prop. 1.35]{AS03}. However, it can be embedded as a non-closed Jordan subalgebra of a JC-algebra. This has as a consequence that the class of JC-algebras governs which identities hold in general special real Jordan algebras.

\begin{proposition}\label{sid}
Let $G$ be a real Jordan polynomial which is satisfied by all JC-algebras. Then $G$ is satisfied by each (identity-)special real Jordan algebra $J$.
\end{proposition}
\begin{proof}
Let $J$ be an identity-special real Jordan algebra. Then the unitization $\tilde{J}$ is a homomorphic image of the free special Jordan algebra $\FSJ_{\R}(S)$ on some set $S$. Hence we only need to show that $\FSJ_{\R}(S)$ satisfies $G$. This will be proved by exhibiting an injective Jordan homomorphism of $\FSJ_{\R}(S)$ into $B(H)_{sa}$ for a suitable real Hilbert space $H$.\\
\ind By construction, $\FSJ_\R(S)$ is the real Jordan subalgebra of $\R \{ S \}$ generated by the set $S$. It suffices to find an injective $\R$-algebra homomorphism $\phi\colon\R\{S\} \to B(H)$ such that $\phi(s) \in B(H)_{sa}$ for each $s\in S$, for then $\phi$ restricts to an injective real Jordan homomorphism from $\FSJ_\R(S)$ into $B(H)_{sa}$.\\
\ind We endow $\R\{S\}$ with the real inner product for which the set of mononomials in $S$ is an orthonormal basis, and let $H$ be the Hilbert space completion of $\R\{S\}$. For every $F \in \R\{S\}$, left multiplication by $F$ on $\R\{S\}$ extends to a bounded operator $\ell_F \in B(H)$. The universal property of the free associative $\R$-algebra gives a unique $\R$-algebra homomorphism $\phi\colon \R \{ S\} \to B(H)$ such that $\phi(s) = \ell_s + \ell_s^*$ for every $s\in S$. The proof is finished by showing that $\phi$ is injective.\\
\ind Note that $H$ inherits from $\R\{S\}$ a grading given by the degree of words. For $s\in S$, the operator $\ell_s$ has degree $1$, hence its adjoint $\ell_s^*$ has degree $-1$. Let $G \in \R\{ S\}$ be non-zero, and let $n\ge 1$ be its degree. Let $G_n$ be the homogeneous component of $G$ of degree $n$. The vector $\ell_G(1)$ and $\phi(G)(1)$ in $H$ have the same homogeneous component of degree $n$, namely $G_n \neq 0$. Therefore $\phi(G) \neq 0$ and $\phi$ is injective.
% Let $X$ be the $\R$-basis of $\R\{S\}$ consisting of all words in $S$. Let $H = \ell^2(X)$ be a Hilbert space having an orthonormal basis $\{e_x: x \in X\}$ in bijection with $X$. Note that $H$ has a natural grading given by the degree of words.
% For every $x\in S$ we define the operator $l_x \in B(H)$ by $\ell_x(e_y) = e_{sy}$, where $xy$ denotes the concatenation of the words $x$ and $y$. For $s\in S$, the operator $\ell_s$ has degree $1$, hence its adjoint $\ell_s^*$ has degree $-1$.
% The universal property of the free associative $\R$-algebra gives a unique ring homomorphism $\pi\colon \R \{ S\} \to B(H)$ such that $\pi(s) = \ell_s + \ell_s^*$ for every $s\in S$. The proof is finished by showing that $\pi$ is injective.\\
% \ind Let $G \in \R\{ X\}$ be non-zero, and let $n\ge 0$ be its degree. Let and let $G_n$ be the homogeneous component of $G$ of degree $n$. Then the homogeneous component of $\pi(G)(e_{1}) \in H$ of degree $n$ is given by $G_n \neq 0$, hence $\pi(G) \neq 0$. This shows that $\pi$ is injective, and concludes the proof.
\end{proof}

\begin{remark}
Each integral Jordan polynomial $F$ which is satisfied by all JC-algebras, is satisfied by each special Jordan algebra $J$ over every commutative ring $R$. (If $2$ is not invertible in $R$, then $J$ is understood to be a quadratic Jordan algebra over $R$ in the sense of \cite[Ch. 1, Def. 3]{Jac69}). Indeed, by restriction of scalars $J$ becomes a special Jordan algebra over $\Z$. Therefore, $J$ is a homomorphic image of $\FSJ_\Z(S)$. Since $\FSJ_\Z(S)$ is a Jordan subalgebra of $\FSJ_\R(S)$, it also embeds in $B(H)_{sa}$. Then $F$ is fulfilled by $B(H)_{sa}$, hence by $\FSJ_\Z(S)$, hence by $J$. 
\end{remark}

\begin{example} One may wonder whether it is possible to embed each special real Jordan algebra in a JC-algebra. The identity-special but not special real Jordan algebra $J := \FSJ([3])/K$ from \Cref{isnots} provides a counterexample. In fact, if $J$ is a Jordan subalgebra of a JB-algebra $A$, then the norm closure $\overline{J}$ of $J$ in $A$ is a JB-subalgebra of $A$. Since $J$ is identity-special, by continuity of the product in $A$, also $\overline{J}$ is identity-special. But then $\overline{J}$ is special by \Cref{JCThm}, forcing $J$ to be special which it is not.
\end{example}

\subsection{JBW-algebras}
In this section, we discuss the Jordan analogue of von Neumann algebras (or more accurately, W*-algebras), called JBW-algebras.\\
\ind Let $M$ be a unital JB-algebra. A linear functional $\phi\colon M \to \R$ on $M$ is called a \emph{state} if $\phi$ has norm $1$ and is positive, i.e.\ $\phi(M_+) \subset \R_+$, equivalently, if $\phi \in M^*$ is bounded with $\lVert \phi \rVert = \phi(1) = 1$. We say that $M$ is \emph{monotone complete} if each upper bounded increasing net $(x_i)_i$ in $M$ has a supremum $x$ in $M$. A state $\phi$ on $M$ is called \emph{normal} if, for each increasing net $(x_i)_i$ in $M$ with supremum $x$ as above, one has $\phi(x) = \lim_i \phi(x_i)$. 

\begin{definition}
A unital JB-algebra $M$ is called a \emph{JBW}\emph{-algebra} if $M$ is monotone complete and the normal states on $M$ separate points.
\end{definition}

Paralleling \Cref{JBEgs}, we have the following examples of JBW-algebras.

\begin{examples}\label{JBWEgs}
(1) An associative JB-algebra $C_0(X)$ is a JBW-algebra if and only if $X$ is a hyperstonean topological space. \\
(2) The self-adjoint part of a von Neumann algebra $A$ is a JBW-algebra $A_{sa}$.\\
(3) Each finite-dimensional or reflexive JB-algebra is a JBW-algebra. In particular, the Albert algebra $\A$ is a JBW-algebra. 
\end{examples}

 %The convex set $K \subset M_*$ of all normal states on $M$ is called the \emph{normal state space} of $M$.
%The $\sigma$\emph{-strong topology} on $M$ is the topology defined by the family of all seminorms $a \mapsto \phi(a^2)^{1/2}$ for $\phi \in K$. Norm convergence implies convergence in the $\sigma$-strong topology. The multiplication in $M$ is jointly $\sigma$-strongly continuous on bounded subsets of $M$ \cite[Prop. 2.4]{AS03}.\\
A JB-algebra $M$ is a JBW-algebra if and only if it is a Banach dual space. In this case, the linear span $M_*$ of the normal states on $M$ is the unique predual of $M$ (up to isometric isomorphism). The weak topology $\sigma(M, M_*)$ on $M$ determined by $M_*$ is called the \emph{$\sigma$-weak topology} on $M$.\\
\ind If $A$ is a JB-algebra, then its bidual $A^{**}$ is a JBW-algebra for a unique product extending the product on $A$ \cite[Prop. 2.4 and Cor. 2.50]{AS03}. Kaplansky's density theorem for JB-algebras \cite[Prop. 2.69]{AS03} asserts that the unit ball of $A$ is $\sigma$-strongly dense in the unit ball of $A^{**}$. \\
\ind A \emph{JW-algebra} is a JB-algebra which is isometrically isomorphic to a $\sigma$-strongly closed Jordan subalgebra of $B(H)_{sa}$. The bidual of a JC-algebra is a JW-algebra \cite[Prop. 2.77]{AS03}. If a JC-algebra is also a JBW-algebra, then it is in fact a JW-algebra \cite[Cor. 2.78]{AS03}. Alfsen--Shultz--Størmer have demonstrated in \cite{ASS78} that each JBW-algebra is isomorphic to $M \oplus C(X, \A)$ for some JW-algebra $M$ and compact Hausdorff space $X$. Since each JB-algebra $A$ is a JB-subalgebra of its bidual JBW-algebra $A^{**}$, this yields the Gelfand--Naimark theorem for JB-algebras which we stated in \Cref{GNJB}.

\section{A generalization of the Shirshov--Cohn theorem}\label{SCSec}

We will now prove the first main result of this article, which is the following generalization of the Shirshov--Cohn theorem for JB-algebras.

\begin{theorem}\label{ShirCohn2}
Let $A$ be a JB-algebra containing two associative Jordan subalgebras $B$ and $C$ such that $A$ is generated as a JB-algebra by $B\cup C$. Then $A$ is a JC-algebra.
\end{theorem}

\begin{proof}
By assumption, the Jordan algebra $D$ generated by $B$ and $C$ is dense in $A$. If $A$ is not isomorphic to a JC-algebra, then there exists a surjective homomorphism $\pi\colon A \to \A$ onto the Albert algebra \cite[Cor. 4.20]{AS03}. Since $\pi$ is norm continuous \cite[Prop. 1.35]{AS03}, it follows that $\pi(D)$ is dense in $\pi(A)$. Because $H_3(\O)$ is finite-dimensional, each of its subspaces is norm closed, so $\pi(D) = \pi(A) = \A$. Therefore $\A$ is algebraically generated by $\pi(B)$ and $\pi(C)$.\\
\ind By the assumption that $B$ and $C$ are associative Jordan algebras, their homomorphic images $\pi(B)$ and $\pi(C)$ are also associative. Since a finite-dimensional associative JB-algebra is isomorphic to $\R^k$, we find that $\pi(B) \cong \R^m$ and $\pi(C) \cong \R^n$ for some integers $m,n\ge 1$ (in fact, comparing ranks gives $m,n\le 3$). Now $\R^k$ can be generated as a Jordan algebra by a single element $(\lambda_1, \ldots, \lambda_k)$ by choosing $\lambda_1,\ldots,\lambda_k$ pairwise distinct. Hence each of $\pi(B)$ and $\pi(C)$ can be generated by a single element. This implies that the Albert $\A$ is generated by two elements, which is absurd by the algebraic Shirshov--Cohn theorem. Therefore $A$ is JC-algebra, proving the theorem.
\end{proof}

\begin{corollary}\label{JWThm} Let $M$ be a JBW-algebra containing two associative Jordan subalgebras $B$ and $C$ such that $M$ is generated as a JBW-algebra by $B\cup C$. Then $M$ is a JW-algebra.
\end{corollary}
\begin{proof}
Let $D$ be the JB-subalgebra of $M$ generated by $B \cup C$, which is $\sigma$-weakly dense in $M$ by assumption. Then $D$ is a JC-algebra by \Cref{ShirCohn2}, hence $D$ satisfies Glennie's identity $G_8$. Since the multiplication in $M$ is separately $\sigma$-weakly continuous, by \Cref{tauLem} its $\sigma$-weak closure $M$ satisfies Glennie's identity $G_8$. But then $M$ is a JC-algebra by \Cref{JCThm}, hence a JW-algebra. 
\end{proof}

%It suffices to show that $M$ satisfies Glennnie's identity, i.e.\ for all $x,y,z\in M$ one has $$H_8(x,y,z)=H_8(y,x,z).$$ Since $D$ is special, this identity is fulfilled if $x$, $y$ and $z$ belong to $D$. According to Kaplansky's density theorem \cite[Prop. 2.69]{AS03} the unit ball of $D$ is $\sigma$-strongly dense in the unit ball of $M$. Because multiplication is jointly $\sigma$-strongly continuous on bounded subsets of $M$ by \cite[Prop. 2.4]{AS03}, it follows that $M$ satisfies Glennie's identity.
%\ind (2) The proof of (2) is the same as (1), except that this time we need to find a surjective homomorphism $\pi\colon M \to H_3(\O)$ which is $\sigma$-weakly continuous, rather than merely norm continuous. According to \cite[Thm. 6.4.1, Rmk. 6.4.3 and Thm. 7.2.7]{HOSt84} there exists a decomposition $M = M_{sp} \oplus M_{exc}$, where $M_{sp}$ is a JW-algebra and $M_{exc} \cong C(X; H_3(\O))$ for some compact Hausdorff space $X$. If there would exist $x\in X$ then $\pi_x \colon M \to M_{exc} \stackrel{\ev_x}{\to} H_3(\Oc)$ is a surjective (WRONG: it is not $\sigma$-weakly continuous) homomorphism. Therefore $H_3(\O)=\pi(M)=\pi(D^{\sigma-wk})\subset \pi(D)^{\sigma-wk} = \pi(D),$ hence $\pi(D) = H_3(\Oc)$. By the same argument as in (1), this is impossible by the Shirshov--Cohn theorem. Therefore $X = \emptyset$ and $M = M_{sp}$ is a JW-algebra.

\begin{remark}
One can also deduce \Cref{JWThm} from \Cref{ShirCohn2} in a similar way as to how in \cite[Lemma 2.3]{HaaHO} part (b) is deduced from part (a). The argument is as follows. If $D$ is a $\sigma$-weakly dense JC-subalgebra of a JBW-algebra $M$, then by \cite[Thm. 2.65]{AS03} there is a normal homomorphism $\pi \colon D^{**} \to M$ extending the inclusion map of $D$ into $M$. Since $D$ is $\sigma$-weakly dense in $M$, it follows that $\pi$ is surjective. So there is a central idempotent $c\in D^{**}$ such that $\pi$ restricts to an isomorphism of $cD^{**}$ onto $M$. Since $D$ is a JC-algebra, $D^{**}$ is a JW-algebra, hence $M \cong cD^{**}$ is a JW-algebra.
\end{remark}

% \begin{remark}

% The proof of the two-element version of the Shirshov--Cohn theorem given in \cite[Thm. 7.2.5]{HOSt84} for the JB-algebra case does not seem to generalize to the JBW-algebras case because a non-special JBW-algebra $M$ does not necessarily admit a $\sigma$-weakly continuous surjective homomorphism onto the Albert algebra $\A$ (for example, $M = C(\Omega; \A)$ where $\Omega$ is a non-empty hyper-Stonean space without isolated points), although it has been cited in the literature for the JBW-algebra case as well.
% %[Iochum, JBW-version used https://projecteuclid-org.leidenuniv.idm.oclc.org/journals/pacific-journal-of-mathematics/volume-122/issue-2/Nonassociative-Lp-spaces/pjm/1102701894.pdf]
% %https://arxiv.org/pdf/1912.01903
% \end{remark}

\section{Macdonald's principle and operator commutativity}\label{McDSec}
Let $J$ be a JB-algebra and let $a,b\in J$. Recall that the multiplication operator defined by $a$ is denoted $T_a\colon J \to J$, $T_a(x) := a \circ x$. One says that $a$ and $b$ \emph{operator commute} in $J$ if $T_a \circ T_b = T_b \circ T_a$ as linear endomorphism of $J$, equivalently, if $a \circ (x \circ b) = (a \circ x ) \circ b$ for all $x\in J$. Note that a priori, operator commutativity of $a$ and $b$ could depend on the ambient JB-algebra $J$. Van de Wetering proved that the formally weaker condition that $a$ and $b$ operator commute in $\JB(a,b)$ implies that $a$ and $b$ operator commute in $J$. Using \Cref{ShirCohn2} we shall prove a generalization of \cite[Theorem]{vdWopCm}, in which we consider a general subset $S\subset J$ rather than $\{a,b\}$.

\begin{theorem}\label{SocThm}
Let $J$ be a JB-algebra and $S$ a nonempty subset of $J$. Assume that all $s,s'\in S$ operator commute in $\JB(s,s')$. Then $\JB(S)$ is an associative JB-algebra whose elements pairwise operator commute in $J$. The analogous assertion holds for JBW-algebras.
\end{theorem}

% \begin{theorem}[van~de~Wetering]\label{vdWc}
% Let $J$ be a JB-algebra and let $a,b\in J$. Then the following statements are equivalent:
% \begin{enumerate}[label={\upshape(\arabic*)}]
% \item $a$ and $b$ operator commute in $JB(a,b)$.
% \item $a$ and $b$ generate an associative JB-algebra of mutually operator commuting elements.
% \end{enumerate}
% \end{theorem}
%Note that implication (2) $\implies$ (1) is trivial. As has been observed in \cite{vdWopCm}, if $A$ is a JC-algebra then implication (2) $\implies$ (1) follows easily from the following lemma.
% \begin{proof}[Proof~of~\Cref{vdWc}]
% Assume (1) holds. The Shirshov--Cohn theorem states that $\JB(a,b)$ is a JC-algebra. Since $a$ and $b$ operator commute, the JC-algebra $JB(a,b)$ is associative according to \Cref{HOpc}.
% Let $x,y \in JB(a,b)$ and let $c\in J$. We need only show that $T_x(T_y(c))=T_y(T_x(c))$. By \Cref{ShirCohn2} the JB-algebra $K = JB(a,b,c)$ generated by $JB(a,b)$ and $JB(c)$ is a JC-algebra. Thus we may consider $K$ embedded as a JC-subalgebra of $A_{sa}$ for some C*-algebra $A$.\\
% \ind We first apply \Cref{HOpc} to the JC-subalgebra $JB(a,b)$ of $A_{sa}$, to infer from our assumption that $a$ and $b$ operator commute in $JB(a,b)$ that they commute in $A$. Therefore, $x,y \in JB(a,b)$ also commute in $A$. Applying \Cref{HOpc} now to the JC-subalgebra $K$ of $A$, it follows that $x$ and $y$ operator commute in $K$. Since $c\in K$, we obtain that $T_x(T_y(c))=T_y(T_x(c))$, as desired.
% \end{proof}
We require the following lemma, asserting that in a JC-algebra $D \subset A_{sa}$, operator commutativity of two elements in $D$ is equivalent to their commuting in the enveloping C*-algebra $A$.

\begin{lemma}[Hanche--Olsen]\label{HOpc}
Let $A$ be a C*-algebra and let $D$ be a JC-subalgebra of $A_{sa}$. Let $x,y \in D$. Then $x$ and $y$ operator commute in $D$ if and only if they commute in $A$, i.e.\
$$T_x \circ T_y = T_y \circ T_x \iff xy = yx.$$
\end{lemma}
\begin{proof}
See \cite[Lemma 5.1]{HaOltp}.
\end{proof}

We will prove \Cref{SocThm} by using \Cref{ShirCohn2} to reduce to the case of a JC-algebra.

\begin{proof}[Proof~of~\Cref{SocThm}]
We prove that each nonempty finite subset $F$ of $S$ generates an associative JB-algebra by induction on the cardinality $|F|$ of $F$. The base case $|F|=1$ is clear. Let $n\ge 2$ and assume that $\JB(F)$ is associative whenever $|F|=n-1$. Now let $F\subset S$ with $|F|=n$ be given. Choose $y \in F$ and let $F':= F \setminus \{y\}$. By the inductive hypothesis, $\JB(F')$ is associative. Since $F = F'\cup \{y\}$, then $\JB(F)$ is generated by the union of the associative JB-algebras $\JB(F')$ and $\JB(y)$. Now \Cref{ShirCohn2} yields that $\JB(F)$ is a JC-algebra, hence $\JB(F) \subset A_{sa}$ for a C*-algebra $A$. We may assume that $\JB(F)$ generates $A$ as a C*-algebra. By assumption, all $a,b\in F$ operator commute in $\JB(a,b)$, hence $a$ and $b$ commute in $A$ by \Cref{HOpc}. Then $A$ is generated by the set $F$ of mutually commuting self-adjoint elements, hence $A$ is an abelian C*-algebra. It follows that $\JB(F) \subset A_{sa}$ is associative, completing the inductive step.\\
\ind It follows that the filtered union $\bigcup_F \JB(F)$, where $F$ ranges over all nonempty finite subsets $F\subset S$, is an associative Jordan algebra. Then the norm closure $\JB(S)$ of $\bigcup_F \JB(F)$ is also associative, by continuity of the product.\\
\ind It remains to show that all $x,y \in \JB(S)$ operator commute in $J$, not just in $\JB(S)$. To this end, let $z \in J$ be arbitrary. We will show that $T_x(T_y(z))=T_y(T_x(z))$. By \Cref{ShirCohn2} the JB-algebra generated by $S \cup \{z\}$ is a JC-algebra, say $\JB(S \cup \{z\}) \subset A_{sa}$ for a C*-algebra $A$. Another two-step invocation of \Cref{HOpc} shows that since $x$ and $y$ operator commute in $\JB(S)$, these elements commute in $A$, hence $x$ and $y$ operator commute in $\JB(S \cup \{z\})$. We conclude that $T_x(T_y(z))=T_y(T_x(z))$. Since $z\in J$ was arbitrary, $T_x \circ T_y = T_y \circ T_x$, that is, $x$ and $y$ operator commute in $J$.
\end{proof}

%\end{proof}
%For any set $X$, the free associative $\R$-algebra $\R\langle X\rangle$ with the reversal involution can be embedded into $\ell^2(W(X))$, where $W(X)$ is the free monoid generated by $X$. Therefore any Jordan 
Using the same idea and method of proof as \Cref{ShirCohn2}, we offer a generalization of Macdonald's principle for JB-algebras, which includes as a special case Macdonald's principle with inverses for a JB-algebra \cite[p. 695]{LewMcC}.

\begin{theorem}\label{McDImp} Let  $F(X_1, \ldots, X_n, Y_1, \ldots, Y_m, Z)$ be a real Jordan polynomial of degree $1$ in $Z$. % \in \FJ_\R(\{X_1,\ldots,X_m,Y_1,\ldots, Y_n,Z\})
Suppose that for each JC-algebra $J$, pairwise operating commuting $n$-tuple $(x_1, \ldots, x_n)$, pairwise operator commuting $m$-tuple $(y_1,\ldots, y_m)$ and element $z$ in $J$ one has $F(x_1,\ldots,x_n,y_1,\ldots,y_m,z)=0$.
Then for each JB-algebra $J$, pairwise operating commuting $n$-tuple $(x_1, \ldots, x_n)$, pairwise operator commuting $m$-tuple $(y_1,\ldots, y_m)$ and element $z$ in $J$ one has $F(x_1,\ldots,x_n,y_1,\ldots,y_m,z)=0$.
\end{theorem}

\begin{remarks}
(1) In view of \Cref{SocThm}, the operator commutativity assumption on the tuple $(x_1,\ldots,x_n)$ is equivalent to $x_i \circ (w \circ x_j) = (x_i \circ w) \circ x_j$ for all $1\le i<j\le n$ and $w\in \JB(x_i,x_j)$. A similar remark is in order for the tuple $(y_1,\ldots,y_m)$.\\
(2) The above formulation of Macdonald's principle is purely JB-algebraic, since it asserts that the pertinent identities need only be verified in the restricted functional-analytic class of JC-algebras, in contrast with the larger algebraic class of all special real Jordan algebras.
\end{remarks}

\begin{proof}[Proof~of~\Cref{McDImp}]
By the Gelfand--Naimark theorem for JB-algebras, \Cref{GNJB}, we may assume that $J = \A$. As in the proof of \Cref{ShirCohn2}, the associative Jordan algebra generated by $x_1, \ldots, x_n$ (resp.\ by $y_1,\ldots, y_m$) can be generated by a single element $x\in \A$ (resp.\ $y \in \A$). Let $G_1,\ldots, G_n, H_1, \ldots, H_m$ be real Jordan polynomials in one indeterminate such that $x_i = G_i(x)$ and $y_j = H_j(y)$ for all $1\le i\le n$, $1\le j\le m$. Consider the Jordan polynomial in three variables $$K(X, Y, Z) := F(G_1(X), \ldots, G_n(X), H_1(Y), \ldots, H_m(Y), Z),$$ which is homogeneous of degree $1$ in $Z$. By assumption, each JC-algebra fulfills $K$. By \Cref{sid}, each special real Jordan algebra fulfills $K$. By Macdonald's principle, $K$ is fulfilled by each real Jordan algebra. It follows that $F(x_1, \ldots, x_n, y_1, \ldots, y_m, z) = K(x, y, z) = 0$. 
\end{proof}

\begin{example} We consider the linearization of Glennie's identity $G_8$ with respect to $Z$; let
$$H'_8(X, Y, Z, W) := \{U_XU_Y(W), Z, X \circ Y\} + \{U_XU_Y(Z), W, X \circ Y\} - 2U_XU_Y(\{Z, X \circ Y, W\}).$$
Let $J$ be a JB-algebra and consider four elements $x, y, z, w \in J$. If one of $\{x,y\}$ operator commutes with one of $\{z,w\}$, then the following identity is valid: $$H'_8(x, y, z, w) = H'_8(y, x, z, w).$$
Indeed, we may assume without loss of generality that $x$ and $w$ operator commute. We apply \Cref{McDImp} to the Jordan polynomial $$F(X_1, X_2, Y_1, Z) := H_8'(X_1, Y_1, Z, X_2) - H_8'(Y_1, X_1, Z, X_2),$$
which is satisfied by all JC-algebras because it is the linearization of Glennie's identity $G_8$ with respect to $Z$. Since $x$ and $w$ operator commute in $J$, we find that
$$H'_8(x, y, z, w) - H_8'(y, x, z, w) = F(x, w, y, z) = 0.$$
Many more examples are produced via the same method of linearizing a Jordan polynomial satisfied by all JC-algebras with respect to one or more of its variables. For example, one could linearize $G_8$ (or $G_9$) also with respect to $X$ or $Y$.

\end{example}

\section{The free JB-algebra generated by two projections}\label{FJ2Sec}
In this section we use the Shirshov--Cohn theorem to prove the existence and give an explicit description of the free unital JB-algebra generated by two projections, defined as follows.

\begin{definition} Let $n$ be a positive integer. A \emph{free unital JB-algebra generated by $n$ projections} is the data of a unital JB-algebra $J$ and an $n$-tuple $(p_i)_{i=1}^n$ of projections in $J$ which satisfies the following universal property. For every unital JB-algebra $T$ and $n$-tuple $(q_i)_{i=1}^n$ of projections in $T$ there exists a unique Jordan homomorphism $\phi \colon J \to T$ such that $\phi(p_i) = q_i$ for every $i \in \{1,2,\ldots,n\}$.
\end{definition}

The free unital JB-algebra generated by $n$ projections is unique up to a unique isomorphism and will be denoted $\FJP(p_1,\ldots,p_n)$, if it exists. One can show its existence using methods from universal algebra, although this falls outside the scope of this article.\\
\ind Raeburn and Sinclair have given an explicit description of the free unital C*-algebra generated by two projections. In conjuction with the Shirshov--Cohn theorem, their result yields the following explicit description of $\FJP(p_1, p_2)$.
%XXX: can the Albert algebra be generated by three projections?
%E.g. $v_1v_1^t$, $v_2v_t^t$, $v_3v_3^t$.
%$v_1 = (1, 0, 0)$, $v_2 = \frac{1}{\sqrt{3}}(i, 1, j)$, $v_3 = \frac{1}{\sqrt{3}}(1, \frac{i+l}{\sqrt{2}}, \frac{j+il}{\sqrt{2}})$

\begin{theorem}\label{FJp2} The free unital JB-algebra generated by two projections exists and is isomorphic to the JC-algebra
\begin{equation}
J := \{f\in C([0, 1], M_2(\R)_{sa}) :  f(0), f(1) \text{ are diagonal}\},
\end{equation}
via an isomorphism carrying the generating projections into the functions
\begin{equation}\label{p1p2}
    p_1(t) := \begin{pmatrix}
    1 & 0\\
    0 & 0
    \end{pmatrix},
    \quad
    p_2(t) := \begin{pmatrix}
        t & \sqrt{t(1-t)}\\
        \sqrt{t(1-t)} & 1-t
    \end{pmatrix}.
\end{equation}
\end{theorem}
\begin{proof}
Raeburn and Sinclair show in \cite{RaeSin} that the free unital C*-algebra $C^*(p_1,p_2)$ generated 
by two projections $p_1$ and $p_2$ may be identified with
\begin{equation*}
    A := \{f\in C([0,1], M_2(\C)) : f(0), f(1)\text{ are diagonal}\}
\end{equation*}
such that the generating projections are given by \eqref{p1p2}. Let $J'$ be the JB-subalgebra of $A_{sa}$ generated by $1$, $p_1$ and $p_2$. Let us first show that the triple $(J', p_1, p_2)$ has the desired universal property. Let $T$ be a unital JB-algebra with two projections $q_1, q_2\in T$. It is to be shown that there exists a unique Jordan homomorphism $\phi\colon J'\to T$ such that $\phi(p_i) = q_i$ for $i=1,2$.\\
\ind Denote by $T'$ the unital JB-subalgebra of $T$ generated by $q_1$ and $q_2$. By the Shirshov--Cohn theorem $T'$ is a JC-algebra, so $T'$ may be embedded into $B_{sa}$ for a unital C*-algebra $B$. The universal property of $A = C^*(p_1,p_2)$ yields a unital $*$-homomorphism $\tilde{\phi}\colon A \to B$ such that $\tilde{\phi}(p_1) = q_1$ and $\tilde{\phi}(p_2) = q_2$. Then the $*$-homomorphism $\tilde{\phi}$ induces a Jordan homomorphism $\tilde{\phi}_{sa}\colon A_{sa} \to B_{sa}$. Now $\tilde{\phi}_{sa}^{-1}(T')$ is a JB-subalgebra of $A_{sa}$ containing $1$, $p_1$, $p_2$, so that $J' \subset \tilde{\phi}_{sa}^{-1}(T')$ and $\tilde{\phi}_{sa}(J') \subset T' \subset T$. By restriction we obtain a unital Jordan homomorphism
$\phi = \tilde{\phi}|_{J'}\colon J' \to T$ with $\phi(p_i)=q_i$ for $i=1,2$. It is unique since $J'$ is generated by $p_1$ and $p_2$. Therefore, $J'$ has the required universal property.\\
\ind We finish by showing that $J'= J$, i.e.\ $J$ equals the JB-subalgebra of $A_{sa}$ generated by $1$, $p_1$ and $p_2$. The containment $J' \subset J$ is clear. Towards showing $J \subset J'$,  note that
\begin{equation}
(U_{p_1}p_2)(t) = \begin{pmatrix}
    t & 0\\
    0 & 0
\end{pmatrix},\quad
(U_{1-p_1}p_2)(t) = \begin{pmatrix}
0 & 0\\
0 & 1-t
\end{pmatrix}.
\end{equation}
By the Stone--Weierstrass theorem the space of real polynomials in $1$ and $t$ is dense in $C([0,1], \R)$. It follows that
$$\left\{\begin{pmatrix} f &0\\
0 & g
\end{pmatrix} : f,g\in C([0,1],\R)\right\}\subset J'.$$
This inclusion and $p_2\in J'$ yield that the set 
$$L := \left\{h \in C([0,1], \R) : \begin{pmatrix} 0 & h\\h & 0\end{pmatrix} \in J'\right\}$$
is a closed ideal of $C([0,1],\R)$ containing the function $h_0(t) := \sqrt{t(1-t)}$. It follows from the Stone--Weierstrass theorem that the inclusion $L \subset \{h \in C([0,1],\R) : h(0) = h(1) = 0\}$ is an equality. This establishes that $J \subset J'$, hence $J= J'$ as remained to be shown.
\end{proof}

\section{JB-algebras generated by two projections and an element are special}\label{FJ3Sec}
The Shirshov–Cohn theorem asserts that a JB- or JBW-algebra is special if it can be generated by
two elements. Somewhat surprisingly, a generating set consisting of \emph{three} projections also forces speciality. Similar to \Cref{ShirCohn2}, we will prove this result as a consequence of the fact that the Albert algebra cannot be generated by three projections. We prove this fact using the Peirce decomposition discussed in \Cref{JdsubSect} and the following lemma.
%\ind We first determine the minimum cardinality of a generating set for a finite-dimensional spin factor. Spin factors have been defined in \Cref{JBEgs}(4).  

\begin{lemma}\label{Uxef}
Let $e$ and $f$ be orthogonal projections in a JB-algebra $J$. Let $x \in J$ be such that $e \circ x = f \circ x = \frac{1}{2}x$. Then $x^2 = U_e(x^2) + U_f(x^2)$ holds with $\lVert U_e(x^2)\rVert = \lVert U_f(x^2)\rVert = \lVert x\rVert^2$.
\end{lemma}
\begin{proof}
Working in the JB-subalgebra $U_{e+f}(J)$, which contains $e$, $f$ and $x$, we may assume that $e+f = 1$. The Peirce multiplication rule \eqref{Pr1} gives that $x^2 = U_e(x^2) + U_f(x^2) \in U_e(J) + U_f(J)$. Using that $x^2$ and $f=1-e$ operator commute, we obtain $$U_x(e) = 2x\circ (x\circ e) - x^2 \circ e = 2x \circ \tfrac{1}{2}x - x^2 \circ e = x^2 \circ (1-e) = x^2 \circ f = U_f(x^2).$$
According to \cite[Lemma 1.30]{AS03} for all $a,b\in J$ it holds that $\lVert U_a(b^2)\rVert = \lVert U_b(a^2)\rVert$. Taking $a=x$ and $b=e = e^2$ yields
$$\lVert U_e(x^2)\rVert = \lVert U_x(e)\rVert = \lVert U_f(x^2)\rVert.$$
Because $U_e(J) + U_f(J)$ is a direct sum of JB-algebras, we conclude that
$$\lVert x\rVert^2 = \lVert x^2\rVert = \max\{\lVert U_e(x^2) \rVert, \lVert U_f(x^2)\rVert\} = \lVert U_e(x^2)\rVert = \lVert U_f(x^2)\rVert.$$
\end{proof}

\begin{theorem}\label{Aapq} The Albert algebra $M_3(\O)_{sa}$ cannot be generated by two projections, a third element and the unit element.
\end{theorem}
\begin{proof}
Let $a,p,q\in \A := M_3(\O)_{sa}$ and suppose that $p$ and $q$ are projections. We shall prove that $\JB(1, a, p, q) \subsetneq \A$. By the spectral theorem \cite[Thm. III.1.1]{FaKor} there exists a Jordan frame $\underline{e} = (e_1, e_2, e_3)$ and $\lambda_1,\lambda_2,\lambda_3\in \R$ such that $a = \sum_{i=1}^3 \lambda_i e_i$. %The Jordan frame $\underline{e}$ induces a Peirce decomposition
%$$\A = \R e_1 \oplus \R e_2 \oplus \R e_3 \oplus J_{12}(\underline{e}) \oplus J_{23}(\underline{e}) \oplus J_{13}(\underline{e}).$$ 
Consider the Peirce decomposition \eqref{PcJ} of $p$ with respect to $\underline{e}$, say
$$p = \mu_1e_1 + \mu_2 e_2 + \mu_3 e_3 + b_{12}+b_{23}+b_{13},$$
with $\mu_1, \mu_2, \mu_3 \in \R$ and $b_{rs}\in \A_{rs}(\underline{e}) := \{x\in \A: e_r \circ x = e_s \circ x = \frac{1}{2}x\}$ for $1\le r<s\le 3.$ After renumbering, we may assume that $$b_{13}=0 \text{ or } b_{12}, b_{23} \neq 0.$$
\ind We can find $s_{12} \in \A_{12}(\underline{e})$ such that $\lVert s_{12}\rVert = 1$ and $b_{12} = \lVert b_{12}\rVert s_{12}$: if $b_{12}\neq 0$ take $s_{12} = \lVert b_{12}\rVert^{-1}b_{12}$; if $b_{12} = 0$ choose an arbitrary $s_{12} \in \A_{12}(\underline{e})$ of norm $1$. Then $s_{12}^2 \in \A_{11}(\underline{e}) + \A_{22}(\underline{e}) = \R e_{1} + \R e_{2}$ according to \eqref{Pr1}, that is, $s_{12}^2 = \nu_1 e_1 + \nu_2 e_2$ for certain $\nu_1, \nu_2 \in \R$. Since $s_{12}^2 \ge 0$ we have $\nu_1, \nu_2 \ge 0$. On the other hand, \Cref{Uxef} gives $|\nu_1|=|\nu_2|=\lVert s_{12}\rVert^2 = 1.$ We arrive at $\nu_1 = \nu_2 = 1$, i.e.\ $s_{12}^2 = e_1 + e_2$. Similarly, there exists $s_{23} \in \A_{23}(\underline{e})$ such that $b_{23} = \lVert b_{23}\rVert s_{23}$ and $s_{23}^2 = e_2 + e_3$.\\
\ind Let $(E_{rs})_{r,s=1}^3$ be the standard matrix units of $M_3(\R)$. According to \cite[Prop. 17.1.1]{McC04}, there exists a Jordan automorphism $\phi\colon \A \to \A$ such that $\phi(e_h)=E_{hh}$ for $1\le h\le 3$ as well as $\phi(s_{12})=E_{12} + E_{21}$ and $\phi(s_{23})=E_{23}+ E_{32}$. We find that
\begin{equation*}
\phi(a) = \begin{pmatrix} \lambda_1 & 0 & 0\\
0 & \lambda_2 & 0\\
0 & 0 & \lambda_3
\end{pmatrix}, \quad
\phi(p) = \begin{pmatrix} \mu_1 & \beta_{12} & \beta_{13} \\
\beta_{12} & \mu_2 & \beta_{23}\\
\beta_{13}^* & \beta_{23} & \mu_3
\end{pmatrix},
\end{equation*}
with $\beta_{12} = \lVert b_{12}\rVert$, $\beta_{23} = \lVert b_{23}\rVert \in \R$ and $\beta_{13} \in \O$. We claim that $\beta_{13} \in \R$. If $\beta_{13} = 0$ this is clear, and otherwise by assumption $\beta_{12} , \beta_{23} \neq 0$. Comparing the $(1,3)$ entries in $\phi(p)=\phi(p^2)=\phi(p)^2$ gives that $$\beta_{13} = \mu_1 \beta_{13} + \beta_{12}\beta_{23} + \beta_{13}\mu_3 = \beta_{12}\beta_{23} + (\mu_1+\mu_3)\beta_{13}.$$ From $(1-\mu_1-\mu_3)\beta_{13}=\beta_{12}\beta_{23} \in \R\setminus \{0\}$ we infer that $\beta_{13} \in \R$. Therefore, we have
$$\JB(1, \phi(a), \phi(p)) \subset M_3(\R)_{sa}.$$
\ind We now consider the projection $\phi(q)$. If $\rk(q) \ge 2$, then on replacing $q$ by $1-q$, which has rank $\rk(1-q)=3-\rk(q)$, we may assume that $\rk(\phi(q)) = \rk(q) \le 1$.  Then $\phi(q)$ is an atom or zero. According to \cite[Prop. A.16]{vGKR} there exists an associative subalgebra $\K \subsetneq \O$ such that $\phi(q) \in M_3(\K)_{sa}$. It follows that
$$\phi(\JB(1, a,p,q)) \subset \JB(M_3(\R)_{sa} \cup \{\phi(q)\}) \subset M_3(\K)_{sa} \subsetneq M_3(\O)_{sa} = \A.$$
Applying $\phi^{-1}$, we conclude that $\JB(1, a, p, q) \subsetneq \A$, as desired.
\end{proof}

The algebraic statement about the Albert algebra in \Cref{Aapq} implies the following speciality result for JB-algebras. Note that Albert--Paige showed in \cite[Cor. 2]{AlbPaige} that the exceptional Albert algebra $\A$ can be generated by three elements.

\begin{theorem}\label{ShirCohn3} Let $A$ be a JB-algebra which can be generated as a JB-algebra by two projections, a third element and possibly the unit element. Then $A$ is a JC-algebra.
\end{theorem}
\begin{proof}
Suppose first that $A$ is unital. Let $a,p,q\in A$ be elements with $p=p^2$ and $q=q^2$ such that $A = \JB(1, a, p, q)$. If $A$ is not a JC-algebra, then by \Cref{JCThm}(4) there exists a surjective unital Jordan homomorphism $\phi\colon A \to \A$ onto the Albert algebra. Note that $\phi(p)$ and $\phi(q)$  are projections in $\A$. Now $\A = \phi(A)$ is generated by $\{1, \phi(a), \phi(p), \phi(q)\}$, which is absurd by \Cref{Aapq}. Therefore, $A$ is a JC-algebra. The non-unital case is similar.
\end{proof}

\begin{corollary} Let $A$ be a JBW-algebra which can be generated as a JBW-algebra by two projections, a third element and possibly the unit element. Then $A$ is a JW-algebra.
\end{corollary}
\begin{proof}
The corollary is deduced from the preceding result as in \Cref{JWThm}.
\end{proof}

\begin{lemma}\label{A4}
The Albert algebra $M_3(\O)_{sa}$ can be generated by four atoms.
\end{lemma}

\begin{proof}
According to \cite[Prop. A.16]{vGKR}, the atoms in $M_3(\O)_{sa}$ are of the form
$$p := \begin{pmatrix}
\lVert x_1\rVert^2 & x_1x_2^* & x_1x_3^*\\
x_2x_1^* & \lVert x_2\rVert^2 & x_2x_3^*\\
x_3x_1^* & x_3x_2^* & \lVert x_3\rVert^2
\end{pmatrix}
$$
where $x_1,x_2,x_3\in \O$ associate, that is, $(x_1x_2)x_3 = x_1(x_2x_3)$ and $\lVert x_1\rVert^2 + \lVert x_2\rVert^2 + \lVert x_3\rVert^2 = 1$.
Set
$$q_1 := \begin{pmatrix}
1 & 0 & 0\\
0 & 0 & 0\\
0 & 0 & 0
\end{pmatrix},\quad 
q_2 := \frac{1}{2} \begin{pmatrix}
1 & 1 & 0\\
1 & 1 & 0\\
0 & 0 & 0
\end{pmatrix}, \quad
q_3 := \frac{1}{3} \begin{pmatrix}
1  & -i & -i\\
i & 1 & 1\\
i & 1 & 1
\end{pmatrix}, \quad
q_4 := \frac{1}{3} \begin{pmatrix}
1 & -j & -l\\
j & 1 & -jl\\
l & jl & 1
\end{pmatrix}.
$$
We will prove that $\{q_1, q_2, q_3, q_4\}$ is a generating set for $M_3(\O)_{sa}$. Let $J := \JB(q_1, q_2, q_3, q_4)$.\\
\ind First we show that $\{E_{11}, E_{22}, E_{33}\} \subset J$. Indeed, $E_{11} = q_1 \in J$, hence $E_{22} = P_0(E_{11})(2q_2) \in J$, which gives $E_{33} = P_0(E_{11} + E_{22})(3q_3) \in J$.\\
\ind Therefore, $J$ contains each off-diagonal Peirce component of $q_h \in J$ with respect to every pair of projections among $E_{11}, E_{22}, E_{33}$, which yields
$$
\left \{\begin{pmatrix}
0 & 1 & 0\\
1 & 0 & 0\\
0 & 0 & 0
\end{pmatrix}, 
\begin{pmatrix}
0 & -i & 0\\
i & 0 & 0\\
0 & 0 & 0
\end{pmatrix},
\begin{pmatrix}
0 & -j & 0\\
j & 0 & 0\\
0 & 0 & 0
\end{pmatrix}, 
\begin{pmatrix}
0 & 0 & 0\\
0 & 0 & 1\\
0 & 1 & 0
\end{pmatrix},
\begin{pmatrix}
0 & 0 & -l\\
0 & 0 & 0\\
l & 0 & 0
\end{pmatrix} \right\}
\subset J.
$$
Then $J$ contains the product of the last two matrices, so
$$\begin{pmatrix}
0 & -l & 0\\
l & 0 & 0\\
0 & 0 & 0
\end{pmatrix} = 2 \begin{pmatrix}
0 & 0 & 0\\
0 & 0 & 1\\
0 & 1 & 0
\end{pmatrix} \circ \begin{pmatrix}
0 & 0 & -l\\
0 & 0 & 0\\
l & 0 & 0
\end{pmatrix} \in J.$$
We conclude that $J$ contains the following generating set of the Albert algebra:
$$
\left \{
E_{11}, E_{22}, E_{33}, 
\begin{pmatrix}
0 & 1 & 0\\
1 & 0 & 0\\
0 & 0 & 0
\end{pmatrix}, 
\begin{pmatrix}
0 & -i & 0\\
i & 0 & 0\\
0 & 0 & 0
\end{pmatrix},
\begin{pmatrix}
0 & -j & 0\\
j & 0 & 0\\
0 & 0 & 0
\end{pmatrix}, 
\begin{pmatrix}
0 & -l & 0\\
l & 0 & 0\\
0 & 0 & 0
\end{pmatrix} \right\} \subset J,
$$
whence $J = M_3(\O)_{sa}$.
\end{proof}

%XXX: Make table with least number of atoms :)
% \begin{remark}
% Curiously, for each $h\in \{1,2,3,4\}$ we have $$\dim \JB(q_1,\ldots,q_h) = 3^{h-1}.$$ In fact, $\JB(q_1) \cong \R$, $\JB(q_1, q_2) \cong M_2(\R)_{sa}$, $\JB(q_1, q_2, q_3) \cong M_3(\C)_{sa}$, $\JB(q_1, q_2, q_3, q_4) \cong M_3(\O)_{sa}$.
% \end{remark}

% XXX: For any $n\ge 4$ and $\K \in \{\R, \C, \H\}$ also $H_n(\K)$ can be generated by $n$ atom. To generate a spin factor of dimension $k+1$, one needs at least $k$ atoms. Thus we obtain a complete answer to the question for each Euclidean Jordan algebra what is the least number of atoms needed to generating it. 

\begin{theorem}\label{FJPn}
Let $n\ge 1$. The free unital JB-algebra generated by $n$-projections $\FJ(p_1, \ldots, p_n)$ is a JC-algebra if and only if $n\le 3$.
\end{theorem}
\begin{proof}
If $n\le 3$, then $\FJ(p_1, \ldots, p_n)$ is generated as a unital JB-algebra by at most three projections $p_1, \ldots, p_n$ and $1$, hence is a JC-algebra by \Cref{ShirCohn3}.\\
\ind Now let $n\ge 4$. By \Cref{A4}, we can choose four atomic projections $q_1,\ldots, q_4$ which generate $\A$. By the universal property of $\FJ(p_1, \ldots, p_n)$ there exists a unique unital Jordan homomorphism $\phi \colon \FJ(p_1, \ldots, p_n) \to \A$ such that $\phi(p_h) = q_h$ for $1\le h \le 4$ and $\phi(p_h) = 0$ for $h > 4$. Then $\phi$ is a surjective homomorphism onto the Albert algebra, hence $\FJ(p_1, \ldots, p_n)$ is non-special by \Cref{JCThm}.
\end{proof}

\begin{remark}\label{FJP3}
Write $C^*(p_1,p_2,p_3)$ for the free unital C*-algebra generated by three projections. As in the proof of \Cref{FJp2}, \Cref{FJPn} may be used to show that the free unital JB-algebra $\FJP(p_1,p_2,p_3)$ generated by three projections can be constructed as the unital JB-subalgebra of $C^*(p_1,p_2,p_3)_{sa}$ generated by $p_1$, $p_2$ and $p_3$.
\end{remark}

\bibliography{references}

@article{Shi56,
 author = {Shirshov, A. I.},
 title = {On special {{\(J\)}}-rings},
 fjournal = {Matematicheski{\u{\i}} Sbornik. Novaya Seriya},
 journal = {Mat. Sb., Nov. Ser.},
 volume = {38},
 pages = {149--166},
 year = {1956},
 language = {Russian},
 zbMATH = {3116667},
 Zbl = {0070.02902}
}

@article{JacMcD,
 author = {Jacobson, Nathan},
 title = {{MacDonald}'s theorem on {Jordan} algebras},
 fjournal = {Archiv der Mathematik},
 journal = {Arch. Math.},
 issn = {0003-889X},
 volume = {13},
 pages = {241--250},
 year = {1962},
 language = {English},
 doi = {10.1007/BF01650071},
 zbMATH = {3323911},
 Zbl = {0204.04002}
}

@article{McD,
 author = {Macdonald, I. G.},
 title = {Jordan algebras with three generators},
 fjournal = {Proceedings of the London Mathematical Society. Third Series},
 journal = {Proc. Lond. Math. Soc. (3)},
 issn = {0024-6115},
 volume = {10},
 pages = {395--408},
 year = {1960},
 language = {English},
 doi = {10.1112/plms/s3-10.1.395},
 zbMATH = {3154015},
 Zbl = {0094.25003}
}

@article{vdWopCm,
title = {Commutativity in {J}ordan operator algebras},
author = {van de Wetering, J.},
fjournal = {Journal of Pure and Applied Algebra},
 journal = {J. Pure Appl. Algebra},
volume = {224},
number = {11},
pages = {106407},
year = {2020},
}

@article{ASS78,
author = {Alfsen, E. M. and Shultz, F. W. and St{\o}rmer, E.},
year = {1978},
pages = {11–56},
title = {A {G}elfand–{N}aimark theorem for {J}ordan algebras},
volume = {28},
fjournal = {Advances in Mathematics},
journal = {Adv. Math.},
doi = {10.1016/0001-8708(78)90044-0}
}

@book{SpVe,
 author = {Springer, T. A. and Veldkamp, F. D.},
 title = {Octonions, {Jordan} algebras and exceptional groups},
 note = {Revised {English} version of the original {German} notes},
 year = {2000},
 publisher = {Springer, Berlin},
}

@book{McC04,
 author = {McCrimmon, K.},
 title = {A taste of {Jordan} algebras},
 issn = {0172-5939},
 isbn = {0-387-95447-3},
 year = {2004},
 publisher =  {Springer, New York, NY},
 
 doi = {10.1007/b97489},
 zbMATH = {2043992},
 Zbl = {1044.17001}
}

@article{HaaHO,
 ISSN = {03794024, 18417744},
 URL = {http://www.jstor.org/stable/24714221},
 author = {Uffe Haagerup and Harald Hanche-Olsen},
 journal = {J. Operator Theory},
 number = {2},
 pages = {343--364},
 publisher = {Theta Foundation},
 title = {Tomita--{T}akesaki theory for {J}ordan algebras},
 urldate = {2026-04-07},
 volume = {11},
 year = {1984}
}

@Book{AS03,
 Author = {Alfsen, E. M. and Shultz, F. W.},
 Title = {Geometry of state spaces of operator algebras},
 ISBN = {0-8176-4319-2},
 Year = {2003},
 Publisher = {Birkh{\"a}user, Boston, MA},
 zbMATH = {1885140},
 Zbl = {1042.46001}
}

@article{RaeSin,
author = {Sinclair, Allan M. and Raeburn, Iain},
journal = {Math. Scand.},
number = {2},
pages = {278-290},
title = {The {C}*-algebra generated by two projections.},
volume = {65},
year = {1989},
}

@book{HOSt84,
 author = {Hanche-Olsen, H. and St{\o}rmer, E.},
 title = {Jordan operator algebras},
 fseries = {Monographs and Studies in Mathematics},
 series = {Monogr. Stud. Math.},
 volume = {21},
 year = {1984},
 publisher = {Pitman, Boston, MA},
 zbMATH = {3893796},
 Zbl = {0561.46031}
}

@article{Behncke,
 author = {Behncke, Horst},
 title = {Hermitian {Jordan} {Banach} algebras},
 fjournal = {Journal of the London Mathematical Society. Second Series},
 journal = {J. Lond. Math. Soc., II. Ser.},
 issn = {0024-6107},
 volume = {20},
 pages = {327--333},
 year = {1979},
 language = {English},
 doi = {10.1112/jlms/s2-20.2.327},
 zbMATH = {3645779},
 Zbl = {0415.46039}
}

@Book{Jac69,
 Author = {Jacobson, N.},
 Title = {Lectures on quadratic {Jordan} algebras},
 Year = {1969},
 Publisher = {Tata Institute of Fundamental Research, Bombay},
 Volume = {45},
 Series = {Lect. Math. Phys.},
 zbMATH = {3399446},
 Zbl = {0253.17013}
}

@Article{JNW,
 Author = {Jordan, P. and von Neumann, J. and Wigner, E. P.},
 Title = {On an algebraic generalization of the quantum mechanical formalism},
 Journal = {Ann. of Math. (2)},
 ISSN = {0003-486X},
 Volume = {35},
 Pages = {29--64},
 Year = {1934},
 
 DOI = {10.2307/1968117},
 zbMATH = {3012646},
 Zbl = {0008.42103}
}

@book{FaKor,
 author = {Faraut, J. and Kor{\'a}nyi, A.},
 title = {Analysis on symmetric cones},
 isbn = {0-19-853477-9},
 year = {1994},
 publisher = {Clarendon Press, Oxford},
 zbMATH = {715155},
 Zbl = {0841.43002}
}

@article {FaFe77,
    AUTHOR = {Faulkner, J. R. and Ferrar, J. C.},
     TITLE = {Exceptional {L}ie algebras and related algebraic and geometric
              structures},
   JOURNAL = {Bull. Lond. Math. Soc.},
  FJOURNAL = {The Bulletin of the London Mathematical Society},
    VOLUME = {9},
      YEAR = {1977},
    NUMBER = {1},
     PAGES = {1--35},
      ISSN = {0024-6093,1469-2120},
   MRCLASS = {17B60},
  MRNUMBER = {444729},
MRREVIEWER = {Bernard\ Kolman},
       DOI = {10.1112/blms/9.1.1},
       URL = {https://doi.org/10.1112/blms/9.1.1},
}

@book {Jac71,
    AUTHOR = {Jacobson, N.},
     TITLE = {Exceptional {L}ie algebras},
    SERIES = {Lect. Notes Pure Appl. Math.},
    VOLUME = {1},
 PUBLISHER = {Marcel Dekker, New York, NY},
      YEAR = {1971},
   MRCLASS = {17.30},
  MRNUMBER = {284482},
MRREVIEWER = {N.\ R.\ Wallach},
}

@article{Kaup02,
 author = {Kaup, W.},
 title = {Bounded symmetric domains and derived geometric structures},
 fjournal = {Atti della Accademia Nazionale dei Lincei. Classe di Scienze Fisiche, Matematiche e Naturali. Serie IX. Rendiconti Lincei. Matematica e Applicazioni},
 journal = {Atti Accad. Naz. Lincei Rend. Lincei Mat. Appl.},
 issn = {1120-6330},
 volume = {13},
 number = {3-4},
 pages = {243--257},
 year = {2002},
 zbMATH = {2217219},
 Zbl = {1098.32008}
}

@book {Sa80,
    AUTHOR = {Satake, I.},
     TITLE = {Algebraic structures of symmetric domains},
    SERIES = {Kan\^o{} Memorial Lectures},
    VOLUME = {4},
 PUBLISHER = {Iwanami Shoten, Tokyo; Princeton University Press, Princeton,
              NJ},
      YEAR = {1980},
     PAGES = {xvi+321},
   MRCLASS = {32-02 (17C35 32Mxx 53C35)},
  MRNUMBER = {591460},
MRREVIEWER = {S.\ Murakami},
}

@book {Up85,
    AUTHOR = {Upmeier, H.},
     TITLE = {Symmetric {B}anach manifolds and {J}ordan {$C\sp
              \ast$}-algebras},
    SERIES = {North-Holland Mathematics Studies},
    VOLUME = {104},
 PUBLISHER = {North-Holland Publishing Co., Amsterdam},
      YEAR = {1985},
      ISBN = {0-444-87651-0},
   MRCLASS = {58B12 (17C35 17C65 32M15 46L70)},
}

@book {Up87,
    AUTHOR = {Upmeier, H.},
     TITLE = {Jordan algebras in analysis, operator theory, and quantum
              mechanics},
    SERIES = {CBMS Regional Conference Series in Mathematics},
    VOLUME = {67},
 PUBLISHER = {American Mathematical Society, Providence, RI},
      YEAR = {1987},
     PAGES = {viii+85},
      ISBN = {0-8218-0717-X},
   MRCLASS = {17C35 (17C65 32M15 46H70 46L70 47D25)},
  MRNUMBER = {874756},
MRREVIEWER = {A.\ Kor\'anyi},
       DOI = {10.1090/cbms/067},
       URL = {https://doi.org/10.1090/cbms/067},
}

@article{ZelPII,
 author = {Zel'manov, E. I.},
 title = {Prime {Jordan} algebras {II}},
 fjournal = {Siberian Mathematical Journal},
 journal = {Sib. Math. J.},
 issn = {0037-4466},
 volume = {24},
 pages = {73--85},
 year = {1983},
 
 doi = {10.1007/BF00968798},
 zbMATH = {3847628},
 Zbl = {0534.17009}
}

@article{AHOS,
 author = {Alfsen, E.M. and Hanche-Olsen, H. and Shultz, F.W.},
 title = {State spaces of {C}*-algebras},
 fjournal = {Acta Mathematica},
 journal = {Acta Math.},
 issn = {0001-5962},
 volume = {144},
 pages = {267--305},
 year = {1980},
 language = {English},
 doi = {10.1007/BF02392126},
 zbMATH = {3717977},
 Zbl = {0458.46047}
}

@article{EffSto79,
 author = {Effros, E.G. and St{\o}rmer, E.},
 title = {Positive projections and {Jordan} structure in operator algebras},
 fjournal = {Mathematica Scandinavica},
 journal = {Math. Scand.},
 issn = {0025-5521},
 volume = {45},
 pages = {127--138},
 year = {1979},
 language = {English},
 doi = {10.7146/math.scand.a-11830},
 url = {https://eudml.org/doc/166668},
 zbMATH = {3712627},
 Zbl = {0455.46059}
}

@article{Sverchkov,
	author = {Sverchkov, S.  R. },
	journal = {Algebra and Logic},
	number = {1},
	pages = {62--88},
	title = {Jordan s-identities in three variables},
	volume = {50},
	year = {2011},
}

@article{Thedy87,
 author = {Thedy, Armin},
 title = {A natural s-identity of {Jordan} algebras},
 fjournal = {Communications in Algebra},
 journal = {Commun. Algebra},
 issn = {0092-7872},
 volume = {15},
 pages = {2081--2098},
 year = {1987},
 language = {English},
 doi = {10.1080/00927878708823523},
 zbMATH = {4019290},
 Zbl = {0627.17010}
}

@article{LewMcC,
 author = {Lewand, Robert E. and McCrimmon, Kevin},
 title = {Macdonald's theorem for quadratic {Jordan} algebras},
 fjournal = {Pacific Journal of Mathematics},
 journal = {Pacific J. Math.},
 issn = {1945-5844},
 volume = {35},
 pages = {681--706},
 year = {1971},
 language = {English},
 doi = {10.2140/pjm.1970.35.681},
 zbMATH = {3345539},
 Zbl = {0217.34504}
}

@book{Topping65,
 author = {Topping, D. M.},
 title = {Jordan algebras of self-adjoint operators},
 fseries = {Memoirs of the American Mathematical Society},
 series = {Mem. Am. Math. Soc.},
 issn = {0065-9266},
 volume = {53},
 isbn = {978-0-8218-1253-2; 978-0-8218-9998-4},
 year = {1965},
 publisher = {American Mathematical Society, Providence, RI},
 language = {English},
 doi = {10.1090/memo/0053},
 zbMATH = {3222371},
 Zbl = {0137.10203}
}

@book{Zhev,
 author = {Zhevlakov, K. A. and Slin'ko, A. M. and Shestakov, I. P. and Shirshov, A. I.},
 title = {Rings that are nearly associative. {Transl}. from the {Russian} by {Harry} {F}. {Smith}},
 fseries = {Pure and Applied Mathematics (Academic Press)},
 series = {Pure Appl. Math., Acad. Press},
 issn = {0079-8169},
 volume = {104},
 year = {1982},
 publisher = {Academic Press, New York, NY},
 language = {English},
 zbMATH = {3765986},
 Zbl = {0487.17001}
}

@article{AlbPaige,
 ISSN = {00029947},
 URL = {http://www.jstor.org/stable/1993420},
 author = {A. A. Albert and L. J. Paige},
 journal = {Trans. Amer. Math. Soc.},
 number = {1},
 pages = {20--29},
 publisher = {American Mathematical Society},
 title = {On a homomorphism Property of certain {J}ordan algebras},
 urldate = {2026-04-24},
 volume = {93},
 year = {1959}
}

@article{Albexc,
 ISSN = {0003486X, 19398980},
 URL = {http://www.jstor.org/stable/1968118},
 author = {A. Adrian Albert},
 journal = {Ann. of Math.},
 number = {1},
 pages = {65--73},
 publisher = {[Annals of Mathematics, Trustees of Princeton University on Behalf of the Annals of Mathematics, Mathematics Department, Princeton University]},
 title = {On a Certain Algebra of Quantum Mechanics},
 urldate = {2026-03-19},
 volume = {35},
 year = {1934}
}

@book{JacStrJA,
 author = {Jacobson, N.},
 title = {Structure theory of {Jordan} algebras},
 fseries = {The University of Arkansas Lecture Notes in Mathematics},
 series = {Univ. Arkansas Lect. Notes Math.},
 volume = {5},
 year = {1981},
 publisher = {The University of Arkansas, Fayetteville, AR},
 language = {English},
 zbMATH = {3773843},
 Zbl = {0492.17009}
}

@book{JacStruRep,
 author = {Jacobson, N.},
 title = {Structure and representations of {Jordan} algebras},
 fseries = {Colloquium Publications. American Mathematical Society},
 series = {Amer. Math. Soc. Colloq. Publ.},
 issn = {0065-9258},
 volume = {39},
 isbn = {0-8218-1039-1; 0-8218-3179-8},
 year = {1968},
 publisher = {American Mathematical Society, Providence, RI},
 language = {English},
 url = {www.ams.org/online_bks/coll39/},
 zbMATH = {3346614},
 Zbl = {0218.17010}
}

@article{vGKR,
 author = {van {G}aans, Onno and Kalauch, Anke and Roelands, Mark},
 title = {Order theoretical structures in atomic {JBW}-algebras: disjointness, bands, and centres},
 journal = {Positivity},
 issn = {1385-1292},
 volume = {28},
 number = {1},
 pages = {54},
 year = {2024},
}

@article{ZelP1,
 author = {Zel'manov, E. I.},
 title = {Prime {Jordan} algebras},
 fjournal = {Algebra and Logic},
 journal = {Algebra Logic},
 issn = {0002-5232},
 volume = {18},
 pages = {103--111},
 year = {1979},
 language = {English},
}

@article{HaOltp,
title={On the Structure and Tensor Products of {J}{C}-Algebras},
volume={35},
DOI={10.4153/CJM-1983-059-8},
number={6},
journal={Canad. J. Math.},
author={Hanche-Olsen, H.},
year={1983},
pages={1059–1074}
}

@article{Cohn_1954,
title={On Homomorphic Images of Special {J}ordan Algebras},
volume={6},
fjournal={Canadian Journal of Mathematics},
journal={Canad. J. Math.},
author={Cohn, P. M.},
year={1954},
pages={253–264}}

@article{Glennie,
author = {C. M. Glennie},
title = {{Some identities valid in special Jordan algebras but not valid in all Jordan algebras.}},
volume = {16},
journal = {Pacific J. Math.},
number = {1},
publisher = {Pacific Journal of Mathematics, A Non-profit Corporation},
pages = {47 -- 59},
year = {1966},
}

\end{document}